\documentclass[a4paper]{amsart}
\usepackage{amsmath,amsthm,mathrsfs}
\usepackage{amssymb}
\usepackage{amscd}
\usepackage[all]{xy}
\usepackage{color}
\usepackage{combelow}
\usepackage{enumerate}
\usepackage{hyperref}
\usepackage[all]{xy}
\usepackage[a4paper,
            bindingoffset=0.2in,
            left=1in,
            right=1in,
            top=1in,
            bottom=1in,
            footskip=.25in]{geometry}

\newtheorem{theorem}{Theorem}
\newtheorem*{claim*}{Claim}
\newtheorem*{summary*}{Summary}
\newtheorem*{MA}{Main Assumptions}
\newtheorem*{MT}{Main Theorem}

\newtheorem{corollary}[theorem]{Corollary}
\newtheorem{definition}[theorem]{Definition}
\newtheorem{example}[theorem]{Example}
\newtheorem{lemma}[theorem]{Lemma}
\newtheorem{proposition}[theorem]{Proposition}
\newtheorem*{remark}{Remark}
\newtheorem*{remarks}{Remarks}

\newcommand{\diffto}{\xrightarrow{\raisebox{-0.2 em}[0pt][0pt]{\smash{\ensuremath{\sim}}}}}

\renewcommand{\gg}{\mathfrak{g}}

\newcommand{\ka}{\mathfrak{k}}
\newcommand{\rmap}{\longrightarrow}

\newcommand{\acts}{\curvearrowright}
\newcommand{\act}{\mathrm{a}}
\newcommand{\dd}{\mathrm{d}}

\newcommand{\la}{\langle}
\newcommand{\ra}{\rangle}

\newcommand{\pr}{\mathrm{pr}}
\newcommand{\X}{\mathfrak{X}}
\newcommand{\Gr}{\mathrm{Gr}}
\newcommand{\hor}{\mathrm{hor}}
\newcommand{\ho}{\mathrm{h}_{\theta}}
\newcommand{\cu}{\mathrm{R}_{\theta}}
\renewcommand{\Form}{\omega_{\theta}}
\newcommand{\Ehc}{H_{\theta}}
\newcommand{\piv}{\pi^{\mathrm{v}}}
\newcommand{\pih}{\pi^{\mathrm{h}}}

\begin{document}
\title{
Normal forms for principal Poisson Hamiltonian spaces}
\author{Pedro Frejlich}
\address{UFRGS, Instituto de Matem\'atica Pura e Aplicada, Porto Alegre, Brasil}
\email{frejlich.math@gmail.com}
\author{Ioan M\u{a}rcu\cb{t}}
\address{Radboud University Nijmegen, IMAPP, 6500 GL, Nijmegen, The Netherlands}
\email{i.marcut@math.ru.nl}
\begin{abstract}
We prove a normal form theorem for principal Hamiltonian actions on Poisson manifolds around the zero locus of the moment map. The local model is the generalization to Poisson geometry of the classical minimal coupling construction from symplectic geometry of Sternberg and Weinstein. Further, we show that the result implies that the quotient Poisson manifold is linearizable, and we show how to extend the normal form to other values of the moment map.
\end{abstract}
\maketitle
\tableofcontents
\setcounter{tocdepth}{1}
\section*{Introduction}

%In this paper we study the local structure of Hamiltonian actions on Poisson manifolds, and we prove extension of the classical normal form theorems in symplectic geometry to the Poisson setting.

%Let $G$ be a Lie group with Lie algebra $\mathfrak{g}$. 

% with $G$-equivariant , if there 

A \emph{Poisson Hamiltonian space}
\begin{equation}\label{Poisson_Ham_space}
G\acts (M,\pi)\stackrel{\mu}{\rmap}\gg^*
 \end{equation}
consists of: 
\begin{itemize} 
\item a Poisson manifold $(M,\pi)$, on which a Lie group $G$ acts by Poisson diffeomorphisms,
\item an equivariant map $\mu:M\to \gg^*$, called the \emph{moment map}, where $\gg$ Lie algebra of $G$, 
\end{itemize}with the property that the infinitesimal action $\act_M(v) \in \mathfrak{X}(M)$ of an element $v \in \gg$ is the Hamiltonian vector field corresponding to the function $\langle \mu,v\rangle\in C^{\infty}(M)$:
\[
 \act_M(v) = \pi^{\sharp}\dd \langle \mu,v\rangle.
\]

We prove a normal form theorem for a Poisson Hamiltonian space around the zero locus of the moment map $P:=\mu^{-1}(0)$, under the assumption that the action is \emph{proper and free}. In that case, $P$ is the total space of a principal $G$-bundle $p_S:P \to S:=\mu^{-1}(0)/G$ and, by the Marsden-Weinstein reduction in the Poisson setting \cite{Marsden_Ratiu,Marsden_Weinstein}, its base inherits the structure of a Poisson manifold $(S,\pi_S)$. In the language of Dirac geometry (see e.g.\ \cite{Bursztyn,Meinrenken}), $\pi_S$ is determined by the equality of the pullbacks to $P$ of the two Poisson structures $\pi$ and $\pi_S$, viewed as Dirac structures via their graphs:
\[
 i^!\mathrm{Gr}(\pi) = p_S^!\mathrm{Gr}(\pi_S),
\]where $i:P \to M$ stands for the inclusion. The local model for \eqref{Poisson_Ham_space} is built using a principal connection $\theta\in \Omega^1(P;\mathfrak{g})$ on the principal $G$-bundle $p_S:P\to S$, and is given by the Dirac structure
\begin{equation}\label{Dirac model}
 L_{\widetilde{\theta}}:=\mathcal{R}_{-\dd\widetilde{\theta}}\, p^!\mathrm{Gr}(\pi_S) \quad \textrm{on}\quad  P\times \gg^*,
\end{equation}
i.e., it is the pullback Dirac structure 
$ p^!\mathrm{Gr}(\pi_S)$, where $p=p_S\circ \pr_1:P\times \gg^*\to S $, gauged transformed by the 2-form $-\dd \widetilde{\theta}$, where  
\begin{align*}
&\widetilde{\theta}\in \Omega^1(P\times \mathfrak{g}^*), &\widetilde{\theta}_{(x,\xi)}:=\textrm{pr}_1^*\langle \xi, \theta_x\rangle.
\end{align*}
Then, for the diagonal action of $G$, we have that
\begin{equation}
G\acts (P \times \gg^*,L_{\widetilde{\theta}})\stackrel{\mathrm{pr}_2}{\rmap}\gg^*
 \end{equation}is a Hamiltonian Dirac space, in the sense that, for all $v\in\gg$,
\begin{align*}%\label{Dirac moment}
& \act_{P \times \gg^*}(v)+\dd \langle \mathrm{pr}_2,v\rangle \in \Gamma(L_{\widetilde{\theta}}).
\end{align*}

Our main result can be stated as follows (an improvement will be presented later):

\begin{MT}\label{thm : zero set}
Consider a proper and free Poisson Hamiltonian space
\eqref{Poisson_Ham_space}. Then the Dirac manifold $(P \times \gg^*,L_{\widetilde{\theta}})$ is the graph of a Poisson structure $\pi_S^{-\dd\widetilde{\theta}}$ in an invariant neighborhood $U$ of $P$, and the Hamiltonian Poisson spaces
\begin{align*}
 & G\acts (M,\pi)\stackrel{\mu}{\rmap}\gg^*, & G\acts (U,\pi_S^{-\dd\widetilde{\theta}})\stackrel{\mathrm{pr}_2}{\rmap}\gg^*
\end{align*}are isomorphic around $P$, by a $G$-equivariant diffeomorphism which is the identity on $P$.
\end{MT}

In the symplectic case $(M,\pi=\omega^{-1})$, the Poisson structure on $S$ is also symplectic $\pi_S=\omega_S^{-1}$, and it is obtained as the classical Marsden-Weinstein reduction. In this case, the Dirac structure $L_{\widetilde{\theta}}$ on $P\times\gg^*$ is the graph of the closed 2-form $p^*\omega_S-\dd\widetilde{\theta}$, which is the classical ``coupling construction" due to Sternberg and Weinstein \cite{Sternberg, Weinstein}. The analog of our Main Theorem is known in symplectic geometry \cite[Theorem 6.1]{Meinrenken_notes}, however, the original statement is difficult to trace back: it would follow from a version of the coisotropic embedding theorem \cite{Gotay} in the $G$-equivariant setting; the content of \cite[Chapters 35-40]{Guillemin_Sternberg} comes very close to it; and the exact statement can be found in \cite[Proposition 5.2]{Jeffrey_Kirwan} -- for finding this reference, we would like to thank Eckhard Meinrenken.

The setting of the Main Theorem implies that $M/G$ has a Poisson structure $\pi_{M/G}$, for which the quotient map $M\to M/G$ and the inclusion $S\to M/G$ are Poisson maps. The Main Theorem provides a normal form for $(M/G,\pi_{M/G})$ around the Poisson submanifold $S$. In the symplectic case, the quotient Poisson manifold $(P\times\gg^*,p^*\omega_S-\dd \widetilde{\theta})/G$ has been studied already in \cite{Montgomery}, but only later it has been identified to serve as the first order local model around the symplectic leaf $(S,\pi_S)$ \cite{CrMa12}, and as the integrable version of Vorobjev's local model \cite{Vorobjev01,Vorobjev05}. Around more general Poisson submanifolds, a first order local model was developed only recently \cite{FeMa22}, and our Main Theorem implies:

\begin{corollary}\label{corollary:linearizable}
Under the assumptions of the Main Theorem, the Poisson manifold $(M/G,\pi_{M/G})$ is linearizable around the Poisson submanifold $(S,\pi_S)$, in the sense of \cite{FeMa22}.
\end{corollary}

In the last section, we discuss the local structure around more general values of the moment map. Namely, if \eqref{Poisson_Ham_space} is a principal Hamiltonian $G$-space, to any $\lambda\in \mu(M)$, we associate the data of a principal $G_{\lambda}$-bundle over a Poisson manifold:
\begin{equation}\label{eq:datum}
G_{\lambda}\acts P_{\lambda}\to (S_{\lambda},\pi_{S_{\lambda}}),
\end{equation}
where $P_{\lambda}:=\mu^{-1}(\lambda)$, $S_{\lambda}:=\mu^{-1}(\lambda)/G_{\lambda}$ and the Poisson structure is determined by the condition
\[p_{S_{\lambda}}^!\mathrm{Gr}(\pi_{S_{\lambda}})=i^!\mathrm{Gr}(\pi),\]
where $p_{S_\lambda}:P_{\lambda} \to S_{\lambda}$ is the projection and $i:P_{\lambda}\hookrightarrow M$ is the inclusion.

\begin{theorem}
In the setting of the Main Theorem, consider a value $\lambda \in \gg^*$ of $\mu$ which admits an affine slice $j:\gg_{\lambda}^* \to \gg^*$ for the $G$-action (see Section \ref{sec : Other values of the moment map} for a more details). Then the Poisson Hamiltonian space \eqref{Poisson_Ham_space} is determined, up to isomorphism, in a $G$-invariant neighborhood of $\mu^{-1}(\lambda)$, by the data of the principal $G_{\lambda}$-bundle $P_{\lambda}$ over the Poisson manifold $(S_{\lambda},\pi_{S_{\lambda}})$, from \eqref{eq:datum}. 

More precisely, a $G$-invariant neighborhood of $P_{\lambda}=\mu^{-1}(\lambda)$ in the Poisson Hamiltonian space \eqref{Poisson_Ham_space} is isomorphic to a $G$-invariant neighborhood of $P_{\lambda}$ in the Dirac Hamiltonian space:
\begin{align*}\label{eq : model for other values}
G \acts \left((G\times P_{\lambda}\times \gg_{\lambda}^*)/G_{\lambda},\overline{D}\right)\stackrel{\overline{\mu}}{\rmap} \gg^*,
\end{align*}
where $\overline{\mu}[g,x,\xi]=g\cdot j(\xi)$, and the pullback of $\overline{D}$ to $G\times P_{\lambda}\times \gg_{\lambda}^*$ is the Dirac structure
\[
\mathcal{R}_{\dd \gamma }p_{S_{\lambda}}^!\mathrm{Gr}(\pi_{S_{\lambda}}),
\]
where the 1-form $\gamma$ is built from a principal $G_{\lambda}$-connection $\theta_{\lambda}$ on $P_{\lambda}$, as follows: % and the left-invariant Maurer-Cartan form $\theta_{G}$ on $G$
\[\gamma(V,U,\eta)_{(g,x,\xi)}:=\la j(\xi),(g^{-1})_*(V)\ra-\la \xi,\theta_{\lambda}(U)\ra,\quad (V,U,\eta)\in T_gG\times T_{x}P_{\lambda}\times T_{\xi}\gg^*_{\lambda}.\]
\end{theorem}

\subsection*{Conventions}

We follow sign and Lie group conventions of \cite[Appendix A]{MariusRuiIonut}. For Dirac structures, we follow the notation of \cite{Meinrenken}; thus the standard Courant algebroid $TM \oplus T^*M$ is denoted by $\mathbb{T}M$, pullbacks are denoted $\varphi^!(L)$ and pushforwards by $\varphi_!(L)$. We also denote by
\begin{align*}
    & \mathcal{R}_{\omega} : \mathbb{T}M \diffto \mathbb{T}M, & \mathcal{R}_{\omega}(a)=a+\iota_{\pr_T(a)}\omega
\end{align*}the orthogonal linear map of \emph{gauge transformation} by a two-form $\omega \in \Omega^2(M)$, and by
\begin{align*}
    & \mathcal{R}_{\pi} : \mathbb{T}M \diffto \mathbb{T}M, & \mathcal{R}_{\pi}(a)=a+\pi^{\sharp}(\pr_{T^*}(a))
\end{align*}the ``gauge transformation" by a bivector $\pi \in \mathfrak{X}^2(M)$ (this latter notation is less standard).

\section{Preliminary remarks}

Consider a Hamiltonian Poisson space:
\begin{equation*}
G\acts (M,\pi)\stackrel{\mu}{\rmap}\gg^*.
\end{equation*}
Our goal is to describe the local behavior of this space around $P=\mu^{-1}(0)$. The weakest assumptions that will allow us to conclude interesting properties are the following:
\begin{MA} $\phantom{00}$
\begin{itemize}
\item The action of $\mathfrak{g}$ is free;
\item The action of $G$ is proper.
\end{itemize}
\end{MA}

Both these conditions are $G$-invariant and open; that the second condition is open follows from the stronger version of the slice theorem which assumes properness only pointwise (see e.g.\ \cite{Ortega_Ratiu, Palais}). Therefore, if these conditions are assumed to hold along $P=\mu^{-1}(0)$, it follows that they hold on a $G$-invariant neighborhood of $P$, which can then replace $M$. Henceforth, we will assume that the main assumptions are satisfied.

\subsection*{The moment map as a submersion by Dirac manifolds}

\begin{lemma}\label{lemma:restriction:to:fibers}
 If the action of $\gg$ is free, then the restriction of the moment map $\mu$ to each symplectic leaf of $(M,\pi)$ is a submersion, and each fibre of $\mu$ is a Dirac transversal.
\end{lemma}
\begin{proof}
 Let $x \in M$ and $(S,\omega)$ be the symplectic leaf of $\pi$ through $x$. By the moment map condition,
 \[
  \act_M(v)=\pi^{\sharp} \circ \mu^*(v),
 \]where $v \in \gg$ is seen as a constant one-form on $\gg^*$, the infinitesimal action at $x$ takes values in $T_xS$:
 \begin{align*}
  & \act_{M,x} : \gg \rmap T_xS,
 \end{align*}and this is the composition of
 \[
  \gg=T^*_{\mu(x)}\gg^* \stackrel{\mu^*}{\rmap} T^*_xM \rmap T_x^*S
 \]with the isomorphism $\omega_{\flat}^{-1}:T_x^*S \diffto T_xS$. Hence $\act_{M,x}$ is injective iff the dual composition
 \[
  T_xS \rmap T_xM \stackrel{\mu_*}{\rmap} T_{\mu(x)}\gg^*
 \]
 is surjective. Otherwise said, $\gg \acts M$ is free iff the restriction of $\mu$ to each symplectic leaf is a submersion. This implies that the fibres of the submersion $\mu:M \to \gg^*$ are Dirac transversals --- that is, they meet the symplectic leaves transversely.
\end{proof}

The lemma implies that $P:=\mu^{-1}(0)$ inherits by pullback along the inclusion $i:P \to (M,\pi)$ a smooth Dirac structure, which will be denoted \[L_P:=i^!\mathrm{Gr}(\pi)=\{\pi^{\sharp}(\alpha) + i^*(\alpha) \ | \  \alpha\in T^*M|_{P}, \ \pi^{\sharp}(\alpha)\in TP\}\subset \mathbb{T}P.\]
Geometrically, the presymplectic leaves of $L_P$ are the connected components of the intersections:
\[(S\cap P, \omega|_{S\cap P}),\]
where $(S,\omega)$ is a symplectic leaf of $\pi$. Note that $(S,\omega)$ is a Hamiltonian space for the connected component $G^0$ of the identity. By Lemma \ref{lemma:restriction:to:fibers}, the moment map $\mu|_S:S\to \gg^*$ has $0$ as a regular value. Therefore, the standard argument in symplectic geometry \cite[Theorem 1]{Marsden_Weinstein} shows that the kernel of the two-form $\omega|_{S\cap P}$ is spanned by the $\gg$-action:
\begin{equation}\label{eq:ker:restr}
\ker(\omega|_{S\cap P})_x=\{\act_{P,x}(v) \ | \  v\in \mathfrak{g}\},
\end{equation}
where $\act_P=\act_M|_P$. Globally, this can be phrased as follows: 
\begin{lemma}
The kernel of the Dirac structure $L_P$ is given by the image of the infinitesimal action:
\begin{equation}\label{eq:ker:L_P}
L_P\cap T P=\mathrm{Im} (\act_P).
\end{equation}
\end{lemma}

\begin{remark}\rm
A submersion $p:\Sigma \to N$ whose total space $\Sigma$ has a Lie algebroid $A$ in which each fibre $p^{-1}(x)$ sits as a Lie algebroid transversal is called \emph{a submersion by Lie algebroids} \cite{Subm}:
\begin{align}\label{eq:SLA}
A \Longrightarrow \Sigma \stackrel{p}{\rmap} N.
\end{align}
This condition can be characterized by the surjectivity of the map
\begin{align*}
    p_* \circ \rho_A : A \rmap p^*(TN),
\end{align*}
where $\rho_A$ denotes the anchor of $A$. Every submersion by Lie algebroids admits an \emph{Ehresmann connection} --- that is, a linear splitting $\mathrm{hor}: p^*(TN) \to A$ of the map above. As shown in \cite{Subm}, such connections can be used to trivialize (\ref{eq:SLA}) locally. 

When the Lie algebroid is given by a Dirac structure, we shall speak of a \emph{submersion by Dirac manifolds}. By Lemma \ref{lemma:restriction:to:fibers}, under our Main Assumptions, we have a submersion by Dirac manifolds: 
\begin{equation}\label{sdm}
 \mathrm{Gr}(\pi) \Longrightarrow M \stackrel{\mu}{\rmap} \gg^*
\end{equation}
To prove the normal form theorem, we will adapt the arguments from \cite{Subm} to this setting, under the additional constraint that the construction needs to be made $G$-invariant. 
\end{remark}

\begin{lemma}
Given a submersion by Dirac manifolds $L_{\Sigma} \Longrightarrow \Sigma \to N$ and any smooth map $f:P \to N$, there is a pullback submersion by Dirac manifolds
\[
\overline{f}^!(L_{\Sigma}) \Longrightarrow f^*(\Sigma) \to P,
\]
where the pullback submersion $f^*(\Sigma)= P \times_N \Sigma \to P$ is equipped with the pullback of $L_{\Sigma}$ by 
\begin{align*}
    & \overline{f}: f^*(\Sigma) \to \Sigma, & \overline{f}(y,x) = x.
\end{align*}
Moreover, any Ehresmann connection $\hor:\Sigma\times_N TN \to L_{\Sigma}$ induces a pullback Ehresmann connection $f^!\hor : f^*(\Sigma) \times_P TP \to \overline{f}^!(L_{\Sigma})$, determined by the condition that $f^!\hor$ and $\hor f_*$ are $\overline{f}$-related.
\end{lemma}
\begin{proof}
For Lie algebroids submersions, this was proven in \cite{Subm}, and we briefly explain how to reduce to this case.
Note first that $\overline{f}$ is transverse to the Dirac structure $L_{\Sigma}$. This implies that the pullback Lie algebroid $T(f^*\Sigma)\times_{T\Sigma}L_{\Sigma}$ is canonically isomorphic to the Lie algebroid underlying the pullback Dirac structure $f^{!}(L_{\Sigma})$, via the map:
\begin{equation}\label{identification}
T(f^*\Sigma)\times_{T\Sigma}L_{\Sigma}\diffto f^{!}(L_{\Sigma}),\qquad\qquad  (v,v_1+\xi)\mapsto v+\overline{f}^*\xi.
\end{equation}
By \cite[Lemma 3]{Subm}, the left-hand side is a submersion by Lie algebroids:  \[T(f^*\Sigma)\times_{T\Sigma}L_{\Sigma}\Longrightarrow f^*(\Sigma) \to P\]
and this implies the first part. At the level of Lie algebroids, the pullback connection was constructed in \cite{Subm}. Under the isomorphism \eqref{identification}, this corresponds to the connection from the statement. 
\end{proof}

\subsection*{The equivariant Ehresmann connection}

The local model of the Hamiltonian Poisson space \eqref{Poisson_Ham_space} depends on an auxiliary choice of a \emph{principal connection}, i.e., a $\gg$-valued one-form $\theta \in \Omega^1(M;\gg)$, satisfying, for all $v \in \gg$ and $g \in G$, the equations:
 \begin{align}\label{principal:connection}
  & \theta(\act_M(v)) = v, & g\cdot \theta = \mathrm{Ad}_g \circ \theta.
 \end{align}
Although only the infinitesimal action is assumed to be free, we have that: 

\begin{lemma}\label{lemma:existence:princ:conn}
Under the Main Assumptions, $M$ admits principal connections. 
\end{lemma}
\begin{proof}
The hypothesis that the action is proper implies that $M$ has a $G$-invariant Riemannian metric $\langle \cdot,\cdot\rangle$ \cite[Theorem 4.3.1]{Palais}. The hypothesis that $\gg$ acts freely identifies $\gg\times M $ as a subbundle of $TM$:
 \begin{align*}
   \gg \times M \ni (v,x) \mapsto \act_{M}(v)_x \in T_xM.
 \end{align*} The one-form $\theta \in \Omega^1(M;\gg)$ arising from the projection $TM \to \gg\times M $ along the orthogonal complement to $\gg\times M$ satisfies the conditions \eqref{principal:connection}.
\end{proof}

\begin{lemma}\label{lem : equivariant Ehresmann}
For a principal connection $\theta \in \Omega^1(M;\gg)$, %interpreted as a bundle map $\theta : TM \to \mu^*(T^*\gg^*)$, the rule
the rule
\begin{align*}
 & \mathrm{hor}_x : T_{\mu(x)}\gg^* \to \Gr(\pi)_x, & \mathrm{hor}_x(\xi) = -\pi^{\sharp}\langle \xi, \theta_x\rangle - \langle \xi, \theta_x\rangle
\end{align*}is a $G$-equivariant Ehresmann connection on the submersion by Dirac manifolds \eqref{sdm}.
\end{lemma}
\begin{proof}
 For $x \in M$, denote by $\mathrm{h}_x=\pr_T\circ \mathrm{hor}_x$. For all $\xi + v \in T_{\mu(x)}\gg^*\oplus T^*_{\mu(x)}\gg^*$, we have that
 \begin{align*}
  \langle \mu_*\mathrm{h}_x(\xi),v\rangle = -\langle \pi_x^{\sharp}\langle \xi, \theta_x\rangle,\mu^*(v)\rangle =  \langle \xi, \theta_x(\pi_x^{\sharp}\mu^*(v))\rangle =  \langle \xi, \theta_x(\act_{M,x}(v))\rangle = \langle \xi,v\rangle,
 \end{align*} where in the third equality we used the moment map condition, and in the fourth the definition of $\theta$. This implies that $\mathrm{h}$ is a classical Ehresmann connection: $\mu_*\mathrm{h}_x(\xi)=\xi$. The $G$-equivariance of the Ehresmann connection follows from the $G$-equivariance of $\theta$, together with the fact that the $G$-action is Poisson --- which is to say that $\pi^{\sharp}$ is $G$-equivariant: $\pi_{gx}^{\sharp}=g_*\pi_{x}^{\sharp}g^*$.
%  Next, we show that $\mathrm{h}$ is $G$-equivariant: 
%  \begin{align*}
%   \mathrm{h}_{gx}(g_*\xi) & = -\pi_{gx}^{\sharp}\theta^*_{gx}(g_*\xi)= -\pi_{gx}^{\sharp}\theta_{gx}^*(\mathrm{Ad}_{g^{-1}}^*\xi)\\
%   & = -g_*\pi_{x}^{\sharp}g^*\theta_{gx}^*(\mathrm{Ad}_{g^{-1}}^*\xi) = -g_*\pi_{x}^{\sharp}\theta_{x}^*\mathrm{Ad}_{g}^*(\mathrm{Ad}_{g^{-1}}^*\xi)\\
%   & = g_*\mathrm{h}_{x}(\xi)
%  \end{align*} where in the third equality we used that $G$ acts by Poisson diffeomorphisms, and in the fourth the definition of $\theta$.
\end{proof}

\section{The local model}

To construct the local model, we will use the following \emph{local form data}:
\begin{itemize}
    \item An action of $G$ on $P$, which is proper and infinitesimally free; 
    \item A Dirac structure $L_P$ on $P$, with kernel given by the $\gg$-action: $L_P\cap TP=\mathrm{Im}(\act_P)$.
\end{itemize}

The construction depends also on a principal connection $\theta \in \Omega^1(P;\gg)$, as in Lemma \ref{lemma:existence:princ:conn}. The connection enters the picture through the usual one-form on $P\times \gg^*$ it induces: 
\begin{align*}
& \widetilde{\theta}\in \Omega^1(P\times \mathfrak{g}^*), & \widetilde{\theta}_{(x,\xi)}:=\textrm{pr}_1^*\langle\xi, \theta_x\rangle.
\end{align*}
This one-form has the following properties:
\begin{enumerate}[a)]
\item It is $G$-invariant;
\item For all $(x,\xi)\in P\times \gg^*$ and $v\in \gg$, it satisfies:
\[\langle \widetilde{\theta}_{(x,\xi)},\act_{P \times \gg^*}(v)\rangle=\langle \xi,v\rangle;\]
%its interior product with the vector field corresponding to the infinitesimal action $\act_{P \times \gg^*}(v)$ of $v \in \gg$ on $P \times \gg^*$ is $\langle v,\mathrm{pr}_2\rangle$;
\item It vanishes along $P\simeq P \times \{0\}$.
\end{enumerate}

It will be convenient to consider all one-forms satisfying these properties:

\begin{definition}
A one-form $\alpha \in \Omega^1(U)$ defined on a $G$-invariant neighborhood $U$ of $P$ in $P \times \gg^*$ will be referred to as a {\bf generalized connection} if it satisfies properties a) - c) above. 
\end{definition}

Such generalized connections give rise to local models for the Poisson Hamiltonian space \eqref{Poisson_Ham_space}:

\begin{proposition}\label{pro : local model}
Given a local form data $G\acts (P,L_P)$ and a generalized connection $\alpha \in \Omega^1(U)$ on $U \subset P \times \gg^*$, the Dirac structure $L_{\alpha}$, defined as the gauge transform of $\mathrm{pr}_1^!L_P$ by the two-from $-\dd \alpha$
 \begin{align}\label{local model}
  L_{\alpha}:=\mathcal{R}_{-\dd\alpha}\mathrm{pr}_1^{!}L_P,
 \end{align}
\begin{enumerate}[1)]
 \item is $G$-invariant and 
 \begin{align*}
  G \acts (U,L_{\alpha}) \stackrel{\mathrm{pr}_2}{\rmap} \gg^*
 \end{align*}is a Hamiltonian Dirac manifold, in the sense that, for all $v\in \gg$:
 \begin{align*}
  \act_U(v) + \dd\langle  \mathrm{pr}_2, v\rangle \in \Gamma(L_{\alpha});
 \end{align*}
 \item is (the graph of) a Poisson structure $\pi_{\alpha}$ in a $G$-invariant neighborhood of $P$.
\end{enumerate}
\end{proposition}
\begin{proof}
Since all ingredients in its construction are $G$-invariant, so is the Dirac structure $L_{\alpha}$. Moreover, for any $v \in \gg$, since $\act_P(v)\in L_P$, we have that $\act_U(v) \in \mathrm{pr}_1^!L_P$. Thus invariance of $\alpha$, a): 
\[
 \mathscr{L}_{\act_U(v)}\alpha = \dd \iota_{\act_U(v)}\alpha+\iota_{\act_U(v)}\dd\alpha=0,
\]together with condition b) give
\begin{align}\label{eq : dalpha iso}
 \iota_{\act_U(v)}\dd\alpha = -\dd\iota_{\act_U(v)}\alpha = -\dd \langle \mathrm{pr}_2,v\rangle
\end{align}and hence we obtain the moment map condition:
\[
 \act_U(v) - \iota_{\act_U(v)}\dd\alpha = \act_U(v) + \dd\langle \mathrm{pr}_2, v\rangle \in \Gamma(L_{\alpha}).
\]The Poisson locus of $L_{\alpha}$ is given by the open, $G$-invariant set in which the intersection
\begin{align*}
 \mathcal{R}_{-\dd \alpha}\mathrm{pr}_1^!(L_P) \cap TU
\end{align*}is trivial. Let $u\in \mathcal{R}_{-\dd \alpha}\mathrm{pr}_1^!(L_P) \cap TU|_P$. Then there exists $\xi\in T^*P$ such that $\mathrm{pr}_{1*}(u)+\xi\in L_P$ and $\mathrm{pr}_1^*(\xi)=\iota_u\dd\alpha$. We have that, for any $v\in \gg$:
\[0=\langle \act_P(v), \mathrm{pr}_{1*}u+\xi\rangle= \langle \act_P(v), \xi\rangle=\langle \act_U(v), \mathrm{pr}_1^*(\xi)\rangle=\langle \act_U(v), \iota_{u}\dd\alpha\rangle=
 \iota_{u}\dd\langle \pr_2, v\rangle=\langle  \pr_{2*}u,v\rangle,
\]
where the first equality holds because $\mathrm{Im}(\act_P)\subset L_P$, the third because $\pr_{1*}\act_U(v)=\act_P(v)$, and the one before the last by (\ref{eq : dalpha iso}). We conclude that $u\in \ker \pr_{2*}|_P=TP$. Pulling the equality $\pr_1^*(\xi)=\iota_u\dd\alpha$ back to $P$ then gives $\xi =0$, which is to say that $u$ lies in the kernel of $L_P$. Therefore $u$ is of the form $u=\act_U(v)$ for some $v \in \gg$. But because $0 \in \gg^*$ is a fixed point of the coadjoint action, at points of $P \times \{0\} \subset U$, we have that $\act_P(v)=\act_U(v)$ along $P$, and because (\ref{eq : dalpha iso}) implies that $\dd\alpha : \mathrm{Im}\,  \act_U|_P \diffto \mathrm{Im}\, \mathrm{pr}^*_{2}|_P$ is an isomorphism, we conclude that $u=0$. Hence the Dirac structure on $U$ is Poisson along $P$, and hence in a $G$-invariant neighborhood of $P$.
\end{proof}

\begin{definition}\label{def:local:model}
Consider a $G$-invariant neighborhood $U\subset P\times \gg^*$ of $P$ on which the Dirac structure $L_{\alpha}$ corresponds to a Poisson structure $\pi_{\alpha}$. The Poisson Hamiltonian space:
\[G\acts (U,\pi_{\alpha})\stackrel{\mathrm{pr}_2}{\rmap}\gg^*\]
is the \textbf{local model} corresponding to the generalized connection $\alpha$. 
\end{definition}

The following shows uniqueness of the local model around $\mu^{-1}(0)$:

\begin{proposition}\label{prop:local_models_are_the_same}
Let $\alpha \in \Omega^1(U_{\alpha})$ and $\beta \in \Omega^1(U_{\beta})$ be two generalized connections, and let 
\begin{align*}
 & G \acts (U_{\alpha},{\pi}_{\alpha}) \stackrel{\mathrm{pr}_2}{\rmap} \gg^*, & G \acts (U_{\beta},{\pi}_{\beta}) \stackrel{\mathrm{pr}_2}{\rmap} \gg^*
\end{align*}
be the corresponding local models. There exist $G$-invariant neighborhoods $U'_{\alpha} \subset U_{\alpha}$ and $U'_{\beta}\subset U_{\beta}$ of $P$, and an isomorphism of Poisson Hamiltonian spaces which is the identity on $P$:
\begin{align*}
 & \psi : (U'_{\alpha},{\pi}_{\alpha}) \diffto (U'_{\beta},{\pi}_{\beta}), & \mathrm{pr}_2 = \mathrm{pr}_2 \circ \psi.
\end{align*}%which fixes $P$ pointwise.
\end{proposition}

\begin{proof}
Observe first that, on the $G$-invariant common domain $U_{\alpha} \cap U_{\beta}$, the one-forms
\begin{align*}
 \alpha_t:=\alpha+t\gamma, \qquad \text{where} \qquad \gamma:=\beta-\alpha \in \Omega^1(U_{\alpha} \cap U_{\beta}),
\end{align*}are generalized connections for all $t \in [0,1]$. Let $U \subset U_{\alpha} \cap U_{\beta}$ be a $G$-invariant open set in which each Dirac structure $\mathcal{R}_{-\dd\alpha_t}\mathrm{pr}_1^!L_P$ is the graph of a Poisson structure $\pi_t \in \mathfrak{X}^2(U)$, for all $t\in [0,1]$. Then the time-dependent vector field
$\pi_{t}^{\sharp}(\gamma)$
%\[\mathcal{W}_t:=\pi_{t}^{\sharp}(\gamma)\]
is $G$-invariant and vanishes along $P$ (because so does $\gamma$); hence there is an open, $G$-invariant neighborhood of $P$ in which its flow $\psi_t$ is defined for all $t\in [0,1]$. These are $G$-equivariant maps, and by \cite[Lemma 3.4]{Alekseev_Meinrenken_2007} or \cite[Lemma 4]{Frejlich_Marcut} the time-one flow $\psi:=\psi_1$ maps $\pi_0=\pi_{\alpha}$ to $\pi_1=\pi_{\beta}$, wherever the flow is defined. Moreover, since $\gamma$ vanishes on $\act_U(v)$ for all $v\in\gg$, we have
\[0=\left\langle \gamma,\act_U(v) \right\rangle=\left\langle\gamma, \pi_{t}^{\sharp}\mathrm{d}\langle \mathrm{pr}_2,v\rangle\right\rangle=
-\big\langle \pi_{t}^{\sharp}(\gamma), \mathrm{d}\langle \mathrm{pr}_2,v\rangle\big\rangle,\]
and therefore $\mathrm{pr}_2$ is preserved by the flow of $\pi_{t}^{\sharp}(\gamma)$.
\end{proof}

\section{The normal form theorem}

We are now ready to state and prove the main theorem

\begin{MT} Any proper and infinitesimally 
free Poisson Hamiltonian space \eqref{Poisson_Ham_space} is isomorphic around $\mu^{-1}(0)$ with the local model from Proposition \eqref{pro : local model}.
\end{MT}

%\subsection*{Proof of the Main Theorem}
\begin{proof}
We pull back the submersion by Dirac manifolds \eqref{sdm} via scalar multiplication: 
\begin{align*}
 & m : I \times \gg^* \to \gg^*, & m(t,\xi):=t\xi,
\end{align*}
where $I:=[0,1]$. The resulting pullback diagram is
\[
\xymatrix{
 m^*(M) \ar[r]^{\overline{m}} \ar[d]_{\overline{\mu}} & M \ar[d]^{\mu}\\
 I \times \gg^* \ar[r]_m & \gg^*
 }
\]
where 
\[m^*(M) = \{ (x,t,\xi) \in M \times I \times \gg^* \ | \ \mu(x)=m(t,\xi)\}\quad \textrm{and}\quad \overline{m}(x,t,\xi)=x.\] 
% where we regard $m^*(M)$ as the subset
% \[
%  m^*(M) = \{ (x,t,\xi) \in M \times I \times \gg^* \ | \ \mu(x)=m(t,\xi)\},
% \]
% and $\overline{m}$ is the induced pullback map $\overline{m}(x,t,\xi)=x$. 
We obtain the submersion by Dirac manifolds: \begin{equation}\label{induced sdm}
 \overline{m}^!\mathrm{Gr}(\pi) \Longrightarrow m^*(M) \rmap I \times \gg^*.
\end{equation}
For each $t\in I$, this can be pulled back via the map $\gg^*\hookrightarrow I\times \gg^*$, $\xi\mapsto (t,\xi)$  and we obtain a family of submersions by Dirac manifolds. 
For $t=1$, this is the original submersion by Dirac manifolds \eqref{sdm}:
\begin{align*}
 \overline{\mu}_1 = \mu : m_1^*(M) = M  \rmap \gg^*,
\end{align*}
equipped with the Dirac structure $\overline{m}_1^!\mathrm{Gr}(\pi) =\mathrm{Gr}(\pi)$. For $t=0$, the induced submersion is
\begin{align*}
 \overline{\mu}_0 = \mathrm{pr_2} : m_0^*(M) = P \times \gg^* \rmap \gg^*,
\end{align*}equipped with the Dirac structure
\begin{align*}
 \overline{m}_0^!\mathrm{Gr}(\pi) = \mathrm{pr}_1^!i^!\mathrm{Gr}(\pi)= \mathrm{pr}_1^!(L_P).
\end{align*}

On the original submersion by Dirac manifolds \eqref{sdm}, we consider an Ehresmann connection induced by a principal connection $\theta\in \Omega^1(M;\gg)$, as in Lemma \ref{lem : equivariant Ehresmann}. Denote the associated generalized connection by $\widetilde{\theta}\in \Omega^1(M\times \gg^*)$. We equip the submersion by Dirac manifolds \eqref{induced sdm} with the pullback Ehresmann connection. Then, the horizontal lift of the vector field $\tfrac{\partial}{\partial t}\in \mathfrak{X}(I \times \gg^*)$ has the form
\[
 \mathrm{hor}(\tfrac{\partial}{\partial t})=\mathcal{V}-r^*\widetilde{\theta} \in \Gamma(\overline{m}^!\mathrm{Gr}(\pi)),
\]
where $r(x,t,\xi)=(x,\xi)$, and the vector component is given by:
\[\mathcal{V}_{(x,t,\xi)}= \big(-\pi^{\sharp}\langle \xi, \theta_x\rangle,\tfrac{\partial}{\partial t},0\big)\in T_{(x,t,\xi)}m^*(M).\]
It is useful to regard it as a vector field on $M \times I \times \gg^*$ which is tangent to $m^*(M)$. Then one can easily see that the flow of $\mathcal{V}$ has the form:
\begin{equation}\label{eq:form:of:diffeo}
\varphi_t: m^*(M) \supset U_t \to m^*(M),\quad \varphi_t(x,s,\xi)=(\varphi_{t,\xi}(x),s+t,\xi),
\end{equation}
where $U_t$ is the maximal domain of the flow and $\varphi_{t,\xi}$ is the flow on $M$ of the vector field $-\pi^{\sharp}\langle \xi, \theta\rangle $. Along the submanifold $P\times I\times \{0\}\subset m^*(M)$, we have that $\varphi_t(x,s,0)=(x,s+t,0)$, and so the flow is defined up to time $t=1$ on $P=P\times \{0\}\times \{0\}$. Therefore, there is an open set \[P\subset O\subset m_0^*(M)=P\times \gg^*,\]
such that the flow is defined on $O$ for all $t\in I$:
\[\varphi_t:O\diffto \varphi_t(O)\subset m_t^*(M).\]
Moreover, because $\mathcal{V}$ is $G$-invariant, $O$ may be chosen $G$-invariant, and so $\varphi_t$ is $G$-equivariant.

The flow of the section $\mathrm{hor}(\tfrac{\partial}{\partial t})$ of the Dirac structure $\overline{m}^!\mathrm{Gr}(\pi)$ is by $G$-equivariant Courant automorphisms of this Dirac structure (in the sense of \cite[Proposition 2.3]{Gualtieri})
\begin{align*}%\label{eq : alpha}
&\Phi_t:(U_t,\overline{m}^!\mathrm{Gr}(\pi))\to (m^*(M),\overline{m}^!\mathrm{Gr}(\pi)),\\
 & \Phi_t: = \varphi_{t\ast} \circ \mathcal{R}_{-\mathrm{d}\alpha_t}, & \alpha_t:=\smallint_0^t \varphi_s^* r^*\widetilde{\theta} \mathrm{d}s.
\end{align*}
Therefore, we have a $G$-equivariant Dirac-embedding between the fibers over $0$ and $1$ 
\begin{align}\label{eq:the_final_iso}
 \varphi_{1} : \left(O,\mathcal{R}_{-\mathrm{d}\alpha_1}\mathrm{pr}_1^!L_P \right) \hookrightarrow  \left(M,\mathrm{Gr}(\pi) \right)
\end{align}which, by the explicit form \eqref{eq:form:of:diffeo}, intertwines $\mathrm{pr}_2$ and the original moment map $\mu$:
\begin{align*}
 \mu \circ \varphi_{1} = \mathrm{pr}_2.
\end{align*}
To conclude the proof, we show that the left-hand side of \eqref{eq:the_final_iso} is isomorphic to a local model. By Proposition \ref{prop:local_models_are_the_same}, it suffices to show that $\alpha_1\in \Omega^1(O)$ is a generalized connection. This can be deduced from the following properties of the map 
$r\circ \varphi_t|_{O}:O \to M\times \gg^*$, which are readily checked using \eqref{eq:form:of:diffeo}: it preserves the projection $\pr_2$ to $\gg^*$, restricts to the inclusion $P\times \{0\}\hookrightarrow M\times \{0\}$ and is $G$-equivariant. Therefore $\varphi_t^*r^*\widetilde{\theta}|_O$ is a generalized connection on $O$, and so is the average $\alpha_1=\smallint_0^1 \varphi_t^* r^*\widetilde{\theta} \mathrm{d}t|_O$. 
\end{proof}

\subsection*{Other proofs}

The Main Theorem is a normal form around the Dirac transversal $P=\mu^{-1}(0)$, which is additionally compatible with the $G$-action and the moment map. In recent years, several techniques have been developed to obtain local forms around transversals. Our proof is modeled on the ideas from \cite{Subm}, which are closely related to \cite{Bischoff_Bursztyn_Lima_Meinrenken_2020}. An alternative proof, based on \cite{Bursztyn_Lima_Meinrenken}, can be obtained by using an equivariant Ehresmann connection as in Lemma \ref{lem : equivariant Ehresmann} to lift the Euler vector field on $\gg^*$ to a section of $\Gr(\pi)$ which is Euler-like along $P \subset M$, and then invoke \cite[Theorem 5.1]{Bursztyn_Lima_Meinrenken}. These steps lead to the same isomorphism \eqref{eq:the_final_iso} as in the proof above. Yet another proof can be deduced from a $G$-equivariant version of the uniqueness of coisotropic embeddings \cite[Proposition 5.1]{Geudens_2020}, which can be proven using the spray methods of \cite{Frejlich_Marcut,Frejlich_Marcut18}.

\section{The case of a principal action}

Consider now the case of a Poisson Hamiltonian space \eqref{Poisson_Ham_space}, for which the action of $G$ on $M$ is proper and free.  In this case, we have that 
\begin{itemize}
\item $S:=P/G$ has an induced Poisson structure $\pi_S$;
\item $M/G$ has a Poisson structure $\pi_{M/G}$, for which both the quotient map $M\to M/G$ and the inclusion $S\hookrightarrow M/G$ are Poisson maps. 
\end{itemize}

The Main Theorem provides a normal form for $(M/G,\pi_{M/G})$ around the Poisson submanifold $S$, which we will describe here. In the case when $S$ is a symplectic manifold, the Poisson  quotient $(P\times\gg^*,p^*\omega_S-\dd \widetilde{\theta})/G$ has been studied in \cite{Montgomery}, and our discussion will generalize this construction. 

\vspace*{0.2cm}

We specialize the local model from Definition \ref{def:local:model} to the proper and free case. The input data is
\begin{itemize}
\item a principal $G$-bundle $P$ over the Poisson manifold $(S,\pi_S)$.
\end{itemize}
To keep track of notation, we put the many projections involved into a diagram: 
\[
\xymatrix{
P\times \gg^* \ar[r]^{p_G}\ar[d]_{\mathrm{pr_1}}\ar[dr]^{p} & P\times_{G}\gg^* \ar[d]^r\\
P  \ar[r]_{p_S} & S. 
}
\]and recall that
\[p_S^!\mathrm{Gr}(\pi_S)=L_P.\]

Fix a principal bundle connection $\theta$ on $P$. On $P\times \gg^*$, we have the $G$-invariant Dirac structure:
\[ L_{\widetilde{\theta}}:=\mathcal{R}_{-\dd \widetilde{\theta}} \, p^!\mathrm{Gr}(\pi_S).\]
By Proposition \ref{pro : local model}, on some open $G$-invariant neighborhood $U\subset P\times \gg^*$ of $P$, $L_{\widetilde{\theta}}$ corresponds to a Poisson structure $\pi_S^{-\dd \widetilde{\theta}}$. The local model for the Hamiltonian Poisson space is: 
\[G\acts (U,\pi_S^{-\dd \widetilde{\theta}})\stackrel{\mathrm{pr}_2}{\rmap}\gg^*.\]
We will describe the Poisson structure on the quotient, i.e., for which the projection is a Poisson map:
\[p_G: (U,\pi_S^{-\dd \widetilde{\theta}})\rmap (U/G,\pi_0).\]

For this, we introduce some notation. 
The horizontal lift with respect to $\theta$ will be denoted by $\ho$. Denote the curvature of $\theta$ by:
\[\cu(u_1,u_2)_{x}:=\big([\ho(u_1),\ho(u_2)]-\ho([u_1,u_2])\big)_{x}\in \gg,\quad \quad u_1,u_2\in \mathfrak{X}(S),\]
where the right-hand side is a vertical vector in $T_yP$, and so we can identify it with an element of $\gg$. Since $\cu(u_1,u_2):P\to \gg$ is $G$-equivariant, the following 2-form is well-defined: 
\[\Form\in \Omega^2(P\times_G\gg^*),\qquad 
\Form(U_1,U_2)_{p_G(x,\xi)}:=\langle \xi,\cu(r_*(U_1),r_*(U_2))_{x}\rangle
\]
The connection $\theta$ induces a linear connection on the associated bundle $P\times_G\gg^*$, and the corresponding Ehresmann connection will be denoted by
\[\Ehc\subset T(P\times_G\gg^*).\]

Next consider the linear Poisson structure $\pi_{\gg}$ on $\gg^*$. Since $\pi_{\gg}$ is $G$-invariant, it induces a vertical, fiberwise linear Poisson structure on $P\times_G\gg^*$, which will be denoted 
\[\piv\in \mathfrak{X}^2(P\times_{G}\gg^*).\]

We first describe the structure which depends only on the connection:

\begin{proposition}\label{lemma:push:forward:foo}
The Dirac structure $\Gr(-\dd \widetilde{\theta})$ pushes forward via the projection $p_G$ to a Dirac structure $D_{\theta}$ on $P\times_G\gg^*$, for which the quotient map
\[
p_G : \left(P \times \gg^*,\Gr(-\dd \widetilde{\theta})\right) \rmap \left(P \times_G \gg^*,D_{\theta}\right)
\]is a strong, forward Dirac submersion. Moreover, $D_{\theta}$ is explicitly given by:
\begin{equation}\label{eq:Dirac:Vorobjev}
D_{\theta}=
\mathcal{R}_{\Form}(\Ehc) \oplus \mathcal{R}_{-\piv}(\Ehc^{\circ})
%D_{\theta}=\{u-i_u\Form \, |\, u\in \Ehc\}\oplus \{{\piv}^{\sharp}\alpha-\alpha \, |\, \alpha\in \Ehc^{\circ}\}.
\end{equation}
\end{proposition}

\begin{remark}\rm
If follows from \eqref{eq:Dirac:Vorobjev} that the Dirac structure on $D_{\theta}$ is a \emph{coupling} Dirac structure with respect to the projection $r:P\times_{G}\gg^*\to S$, meaning that 
(see, e.g., \cite{BrFe08,DuWa,Vaisman}):
\begin{align*}
    & D_{\theta}\cap (V\oplus V^{\circ})=0,& V:=\ker r_*\subset T(P\times_G\gg^*). 
\end{align*} 
Equality \eqref{eq:Dirac:Vorobjev} also shows that the Vorobjev triple corresponding to $D_{\theta}$ is $(-\piv,\Ehc,\Form)$. The Dirac structure $D_{\theta}$ can be obtained also using the more general construction from \cite[Theorem 3.2]{Wade}.
\end{remark}

\begin{proof} Because $\Gr(-\dd \widetilde{\theta})$ is $G$-invariant, it pushes forward pointwise to $P\times_{G}\gg^*$, so $D_{\theta}:=(p_{G})_{!}\Gr(-\dd \widetilde{\theta})$ is a well-defined family of Lagrangian subspaces. We will prove \eqref{eq:Dirac:Vorobjev}, which implies that $D_{\theta}$ is a smooth Lagrangian subbundle, and being the smooth pushforward of a Dirac structure, it will follow \cite[Proposition 5.9]{Bursztyn} that $D_{\theta}$ is a Dirac structure.

We will use the following notation:
\begin{itemize}
\item $\hat{v}:=(\act_{P}(v),0)\in \mathfrak{X}(P\times \gg^*)$, i.e., 
the infinitesimal action of $v\in \gg$ on $P$;
\item $\mathrm{h}_{P \times \gg^*}(u):=(\mathrm{h}_{P}(u),0)\in\mathfrak{X}(P\times \gg^*)$ the horizontal lift of $u\in \mathfrak{X}(S)$;
\item $\hat{\xi}:=(0,\xi)\in \mathfrak{X}(P\times \gg^*)$ the constant vector field with second component $\xi\in \gg^*$. 
\end{itemize}

Vector fields of these three types span $T(P\times\gg^*)$. In fact, we have a direct sum decomposition: 
\begin{equation}\label{eq:decomposition:T}
T_{p_S(y)}S\oplus \gg \oplus \gg^*\simeq T_{(y,\eta)}(P\times \gg^*),\qquad (u,v,\xi) \mapsto\big(\mathrm{h}_{P \times \gg^*}(u)+\hat{v}+\hat{\xi}\big)\big|_{(y,\eta)}.
\end{equation}
%We have a similar decomposition
%\begin{equation}\label{eq:decomposition:Tast}
%T^*_{p_S(y)}S\oplus \gg \oplus \Ehc^{\circ}\simeq T^*_{(y,\eta)}(P\times \gg^*),\qquad (\zeta,v,\tau) \mapsto\big(p^*(\xi)+\rho_{(y,\eta)}(v)+\tau\big)\big|_{(y,\eta)},
%\end{equation}
Consider
\begin{align*}
    & \rho_{(y,\eta)}:\gg \to T^*_yP \times T^*_{\eta}\gg^*, & \rho_{(y,\eta)}(v):= (\la \pi_{\gg,\eta}(v),\theta_y\ra,v),
\end{align*}and observe that
\[
\la \rho(v), \mathrm{h}_{P \times \gg^*}(u)\ra = 0, \qquad \la \rho(v), \hat{w}\ra = \pi_{\gg}(v,w), \qquad 
\la \rho(v), \hat{\xi} \ra =\la \xi,v\ra.
\]
In particular, $\rho(v) \in \act_{P \times \gg^*}(\gg)^{\circ}\subset T^*(P \times\gg^*)$, since the infinitesimal coadjoint action of $\gg$ on $\gg^*$ is given by
\[\act_{\gg^*}(v)_{\xi}=\la \xi, [v,\cdot]\ra=\pi_{\gg,\xi}^{\sharp}(v),\]and therefore
\[
\la \rho(v), \act_{P \times \gg^*}(w)\ra = \la \rho(v), (\hat{w},\pi_{\gg}^{\sharp}(w))\ra = \pi_{\gg}(v,w)+\pi_{\gg}(w,v) = 0
\]
for all $v,w \in \gg$. Because
\[
\iota_{\mathrm{h}_{P \times \gg^*}(u)}\widetilde{\theta} = \iota_{\hat{\eta}}\widetilde{\theta} = 0, \qquad \iota_{\hat{v}}\widetilde{\theta} = \langle \pr_2,v\rangle
\]we deduce that
\begin{align*}
&-\dd \widetilde{\theta}(\mathrm{h}_{P \times \gg^*}(u),\mathrm{h}_{P \times \gg^*}(u'))_{(y,\eta)}=\langle \eta,  \cu(u,u')\rangle\\
&-\dd \widetilde{\theta}(\hat{v},\hat{w})_{(y,\eta)}=\langle \eta, [v,w]\rangle  \\
&-\dd \widetilde{\theta}(\hat{v},\hat{\xi})_{(y,\eta)}=\langle \xi,v \rangle
\end{align*}
are the only non-zero pairings between the three types of vector fields on $P \times \gg^*$. Therefore
\[
-\iota_{\mathrm{h}_{P \times \gg^*}(u)}\mathrm{d}\widetilde{\theta} = \iota_{\mathrm{h}_{P \times \gg^*}(u)}p_G^*(\Form), \qquad -\iota_{\hat{v}}\mathrm{d}\widetilde{\theta} = \rho(v).
\]The first equality implies that the sections
\begin{align*}
 & \mathrm{h}_{P \times \gg^*}(u)-\iota_{\mathrm{h}_{P \times \gg^*}(u)}\mathrm{d}\widetilde{\theta} \in \Gamma(\mathrm{Gr}(-\dd \widetilde{\theta}))
 & \ho(u)+\iota_{\ho(u)}\Form \in \Gamma(\mathcal{R}_{\Form}(\Ehc))
\end{align*}are $p_G$-related, while the second equality means that
\begin{align*}
 & \hat{v}-\iota_{\hat{v}}\mathrm{d}\widetilde{\theta} \in \Gamma(\mathrm{Gr}(-\dd \widetilde{\theta}))
 & -{\piv}^{\sharp}\rho(v)+\rho(v) \in \Gamma(\mathcal{R}_{-\piv}(\Ehc^{\circ}))
\end{align*}are $p_G$-related, since
\begin{align*}
 (p_G)_*\hat{v} = (p_G)_*(\act_P(v),0) = (p_G)_*(0,-\act_{\gg^*}(v))= (p_G)_*(0,-\pi_{\gg}(v)) = -{\piv}^{\sharp}\rho(v).
\end{align*}This shows that
\[
\mathcal{R}_{\Form}(\Ehc) \oplus \mathcal{R}_{-\piv}(\Ehc^{\circ}) \subset D_{\theta},
\]
which implies equality \eqref{eq:Dirac:Vorobjev}, because both are Lagrangian families. This shows that
\begin{align*}
    p_G : \left(P \times \mathfrak{g}^*,\mathrm{Gr}(-\dd \widetilde{\theta}) \right) \rmap \left(P \times_G \mathfrak{g}^*,D_{\theta} \right)
\end{align*}is forward Dirac submersion, and it is \emph{strong} \cite[Definition 1.11]{ABM}, in the sense that $\ker (p_{G})_* \cap  \mathrm{Gr}(-\dd \widetilde{\theta}) = 0$, as one can read off from the non-zero pairings of $\dd \widetilde{\theta} $ listed above.
\end{proof}

In the next step, the tangential product (or \emph{tensor product} in \cite[Section 2.3]{Gualtieri}) of Dirac structures $D_1$ and $D_2$ will be play an important role:
\[D_1 \star D_2:=\{v+\alpha_1+\alpha_2 \, |\, v+\alpha_i\in D_i\}.\]
In general, $D_1\star D_2$ is a family of Lagrangian spaces, but might fail to be smooth. If it is smooth, then it is automatically a Dirac structure. 

\begin{lemma}
Let $D$ be a coupling Dirac structure for a surjective submersion $p:M\to N$. For any Dirac structure $E$ on $N$, $D\star p^!E$ is a Dirac structure on $M$. Moreover, if $D$ pushes forward to a Dirac structure $F$ on $N$, i.e., if 
\[
p : (M,D) \to (N,F)
\]is a forward Dirac map, then so is
\[
p :(M,D \star p^!E) \to (N,F \star E).
\]
\end{lemma}
\begin{proof}
The Lagrangian family $D\star p^!E$ is smooth because $D$ and $p^!E$ \emph{tangentially transverse}, that is, their anchors $D \to TM$ and $p^!(E) \to TM$ are transverse. This is so because $D$ is transverse to the Dirac structure corresponding to the fibres of $p$, and $p^!E$ contains all vectors tangent to such fibres.

Assume now that $p_!D=F$, for a Dirac structure $F$ on $N$. Every $b \in (F \star E)_{p(x)}$ is of the form $b=f+\xi$, where
\[
f \in F_{p(x)}, \qquad \xi \in T_{p(x)}^*N, \qquad \pr_T(f)+\xi \in E_{p(x)}.
\]Because $p:(M,D) \to (N,F)$ is forward, $f$ is $p$-related to some $d \in D_x$. Therefore $a:=d + p^*\xi \in (D \star p^!E)_x$ is $p$-related to $b$, which is to say that
\[
p_!(D \star p^!(E)) = p_!(D) \star E. \qedhere
\]
\end{proof}

Applying the lemma to the submersion $r:P\times_G\gg^*\to S$, the coupling Dirac structure $D_{\theta}$ on $P\times_G\gg^*$ and the Poisson structure $\pi_S$, we obtain that 
\[D_{\theta}\star r^!\Gr(\pi_S)\]
is a smooth Dirac structure on $P\times_G\gg^*$. Noting that:
\[L_{\widetilde{\theta}}=\mathcal{R}_{-\dd \widetilde{\theta}}\,p^!\Gr(\pi_S)=\Gr(-\dd \widetilde{\theta})\star p_G^!r^!\Gr(\pi_S),\]
the previous lemma implies:
\begin{proposition}\label{proposition_coupling_Poisson}
The projection map is a forward Dirac map 
\[p_G:(P\times \gg^*,L_{\widetilde{\theta}})\to (P\times_G\gg^*, D_{\theta}\star r^!\Gr(\pi_S)),\]
and we have that:
\[D_{\theta}\star r^!\Gr(\pi_S)=
\big\{(\piv)^{\sharp}\alpha-\alpha \ | \  \alpha\in \Ehc^{\circ}\big\}\oplus 
\big\{\ho(\pi_S)^{\sharp}\beta+\beta +  i_{\ho(\pi_S)^{\sharp}\beta}\Form \ | \  \beta\in V^{\circ}\big\}.
\]where $\ho(\pi_S) \in \mathfrak{X}^2(P \times_G \gg^*)$ is the horizontal lift of $\pi_S$.
\end{proposition}

Finally, we obtain an explicit formula for the Poisson structure $\pi_0$:

\begin{corollary}\label{corollary:global:form}
Let $U/G\subset P\times_G\gg^*$ be the open set of elements $z=p_G(y,\eta)$, where the map: 
\[\mathrm{Id}+B(z) :T^*_{p_S(y)}S\to T^*_{p_S(y)}S,\quad  \lambda \mapsto \lambda+\langle \eta, \cu(\pi_S^{\sharp}\lambda,\cdot)_y\rangle\]
is invertible. On $U/G$, $D_{\theta}\star r^!\Gr(\pi_S)$ corresponds to the graph of the Poisson structure \[\pi_{0}:=\pih-\piv\in \X^2(U/G),\]
where    
\[\pih=\ho(\gamma)\in \Gamma(\wedge^2\Ehc),\quad \gamma_{z}^{\sharp}:=\pi_{S}^{\sharp}\circ (\mathrm{Id}+B(z))^{-1}.\]
\end{corollary}

\begin{remark}\rm
Using the calculations from Proposition \ref{lemma:push:forward:foo}, it is easy to see that $\Gr(-\dd \widetilde{\theta})$ is also a coupling Dirac structure for the projection $p:P\times \gg^*\to S$, and its corresponding Vorobjev triple has 2-form $p_G^*\Form$, Ehresmann connection $E_{\theta}=\ker\theta\times 0\subset TP\times T\gg^*$ and vertical Poisson structure $\Pi^{\mathrm{v}}$, which is such that its restriction to each fiber $P_x\times \gg^*$ corresponds to minus the inverse of the canonical symplectic structure on $T^*P_x$, under the dual of the action $(\act_{P_x})^*: TP_x^* \diffto P_x\times \gg^*$. Arguments similar to the ones above show further that $L_{\widetilde{\theta}}$ corresponds to a Poisson structure exactly on the set $U$ described in the corollary above, and this Poisson structure admits a similar description:
\[\pi_S^{-\dd \widetilde{\theta}}=\ho(\gamma)+\Pi^{\mathrm{V}},\]
where $\gamma$ is the same as above, but where $\ho$ now denotes the horizontal lift with respect to $E_{\theta}$.
\end{remark}

\subsection*{Linearization around Poisson submanifolds} 

Corollary \ref{corollary:global:form} implies Corollary \ref{corollary:linearizable} from the Introduction. For this, we will briefly discuss the linear models around Poisson submanifolds from \cite{FeMa22}.

Consider a bundle of Lie algebras $(\ka,[\cdot,\cdot]_{\ka})\to S$ over a Poisson manifold $(S,\pi_S)$. We will call a pair  $(\nabla,U)$ a \emph{coupling data} for $\ka$ and $\pi_S$, if $\nabla$ is a connection on $\ka$, $U\in \Gamma(TS\otimes T^*S\otimes \ka)$ is a tensor field, and they satisfy the structure equations:
\begin{enumerate}
\item[(S1)] the connection $\nabla$ preserves the Lie bracket $[\cdot,\cdot]_{\ka}$, i.e.,
\[\nabla_{X}[\xi,\eta]_{\ka}=[\nabla_{X}\xi,\eta]_{\ka}+[\xi,\nabla_{X}\eta]_{\ka};\]
\item[(S2)] the curvature of $\nabla$ is related to $[U,\cdot]$ as follows:
\[\nabla_{\pi^{\sharp}_S(\alpha)}\nabla_{X}-\nabla_{X}\nabla_{\pi^{\sharp}_S(\alpha)}-\nabla_{[\pi^{\sharp}_S(\alpha),X]}=[U(\alpha,X),\cdot ]_{\ka};\]
\item[(S3)] $U$ satisfies the skew-symmetry condition
\begin{equation*}%\label{eq:U:skew} 
U(\alpha,\pi^{\sharp}_S(\beta))=- U(\beta,\pi^{\sharp}_S(\alpha)).
\end{equation*}
and the ``mixed'' cocycle-type equation:
\begin{align*}
\nabla_{\pi^{\sharp}_S(\alpha)}U(\beta,X)&-\nabla_{\pi^{\sharp}_S(\beta)}U(\alpha,X)+\nabla_{X}U(\alpha,\pi_S^{\sharp}(\beta))\\
&+U(\alpha,[\pi^{\sharp}_S(\beta),X])-U(\beta,[\pi_S^{\sharp}(\alpha),X])=U([\alpha,\beta]_{\pi_S},X),
\end{align*}
\end{enumerate}
for all $X\in \X^1(S)$, $\alpha,\beta\in \Omega^1(S)$, $\xi,\eta\in \Gamma(\ka)$,and where 
\begin{align*}
    & [\cdot,\cdot]_{\pi_S} : \Omega^1(M) \times \Omega^1(M) \to \Omega^1(M), & [\alpha,\beta]_{\pi_S}:=\mathscr{L}_{\pi_S^{\sharp}(\alpha)}\beta-\iota_{\pi_S^{\sharp}(\beta)}\dd\alpha
\end{align*}stands for the bracket induced by $\pi_S$ on one-forms on $S$.

A coupling data induces first of all a Lie algebroid structure on $A:=T^*S\oplus \ka$, with anchor $\rho(\alpha,\xi):=\pi_S^{\sharp}(\alpha)$ and Lie bracket:
\begin{equation*}%\label{equation:bracket:splitting}
[(\alpha,\xi),(\beta,\eta)]:=([\alpha,\beta]_{\pi_S},[\xi,\eta]_{\ka}+\nabla_{\pi_S^{\sharp}(\alpha)}\eta-\nabla_{\pi_S^{\sharp}(\beta)}\xi+U(\alpha,\pi_S^{\sharp}(\beta))).
\end{equation*}
As the formula for the bracket shows, we have a short exact sequence of Lie algebroids
\[0\rmap (\ka,[\cdot,\cdot]_{\ka})\rmap (A,[\cdot,\cdot])\rmap (T^*S,[\cdot,\cdot]_{\pi_S})\rmap 0.\]

Secondly, the coupling data induces a Poisson structure $\pi_0$, which is a \textbf{linear model}. It is defined on the neighborhood of $S$ in $\ka^*$, consisting of points $z\in \ka^*$ where the map $\mathrm{Id}+\langle z, U\rangle:T^*_{p(z)}S\to T^*_{p(z)}S$ is invertible, and is given by
\[\pi_0=\pi_{\ka}+\hor(\gamma),\]
where $\pi_{\ka}$ is the vertical, fiberwise linear Poisson structure on $\ka^*$, $\hor$ denotes the horizontal lift with respect to $\nabla$ and $\gamma\in \wedge^2 TS$ 
\[\gamma_z^{\sharp}=\pi_S^{\sharp}\circ (\mathrm{Id}+\langle z, U\rangle)^{-1}.\]

\begin{example}\label{ex:principal:type}\rm
A principal $G$-bundle $P\to S$ gives rise to a transitive Lie algebroid, $A(P):=TP/G$, with bracket coming from the Lie bracket on $G$-invariant vector fields:
\[\Gamma(A(P))\simeq \mathfrak{X}^1(P)^G.\]
Being transitive means that $A(P)$ fits into a sort exact sequence:
\begin{equation}\label{eq:sesAtiyah}
0\rmap \ka\rmap A(P)\stackrel{\rho}{\rmap} TS\rmap 0,
\end{equation}
where $\rho$ denotes the anchor of $A(P)$. The bundle of isotropy Lie algebras can be canonically identified with the bundle adjoint bundle 
$P\times_G\gg\simeq \ka$, and the map at the level of sections is given by:
\[\xi\in C^{\infty}(P;\gg)^G\mapsto \tilde{\xi}\in \mathfrak{X}(P)^G,\quad \tilde{\xi}_{x}:=\act_P(\xi_x),\]
where we think of sections of $P\times_G\gg$ as $G$-equivariant maps $P\to \gg$, and we regard sections of $\ka$ as $G$-invariant vertical vector fields on $P$. We will use this isomorphism to further identify $\ka=P\times_G\gg$. However, note that under this identification the Lie bracket $[\cdot,\cdot]_{\ka}$ corresponds to the opposite of the fiber-wise Lie bracket on $P\times_G \gg$ coming from $\gg$, i.e., 
\begin{equation}\label{the_other_bracket}
[\xi,\eta]_{\ka,x}=-[\xi_x,\eta_x]_{\gg}.
\end{equation}

Principal connections on $P$ are in 1-to-1 correspondence with splittings of the short exact sequence \eqref{eq:sesAtiyah}. Fix a principal connection $\theta$. The corresponding splitting sends $X\in \mathfrak{X}(S)$ to its horizontal lift $\mathrm{h}_{\theta}(X)$. Under the resulting isomorphism $A(P)\simeq TS\oplus \ka$ the Lie bracket of $A(P)$ takes the form:
\begin{equation*}%\label{equation:bracket:splitting}
[(X,\xi),(Y,\eta)]:=([X,Y],[\xi,\eta]_{\ka}+\nabla_{X}\eta-\nabla_{Y}\xi+\cu(X,Y)),
\end{equation*}
where $\nabla$ is a linear connection on $\ka$ and $\cu\in\Omega^2(S;\ka)$ is the curvature of $\theta$. The Jacobi identity for the bracket implies the following relations for all $X,Y,Z\in \X^1(S)$,  $\xi,\eta\in \Gamma(\ka)$: 
\begin{enumerate}
\item[(P1)] the connection $\nabla$ preserves the Lie bracket $[\cdot,\cdot]_{\ka}$, i.e.,
\[\nabla_{X}[\xi,\eta]_{\ka}=[\nabla_{X}\xi,\eta]_{\ka}+[\xi,\nabla_{X}\eta]_{\ka};\]
\item[(P2)] the curvatures of $\nabla$ and $\theta$ are related by: 
\[\nabla_{X}\nabla_{Y}-\nabla_{Y}\nabla_{X}-\nabla_{[X,Y]}=[\cu(X,Y),\cdot ]_{\ka};\]
\item[(P3)] $\cu$ satisfies the second Bianchi identity: 
\begin{equation*}
\nabla_{X}\cu(Y,Z)+\cu([X,Y],Z)+c.p.=0.
\end{equation*}
\end{enumerate}

The equations above imply that the pair $(\nabla,U)$, where
\[U(\alpha,X):=\cu(\pi_S^{\sharp}(\alpha),X)\]
is a coupling data for the bundle of Lie algebras $\ka=P\times_G\gg$ and $(S,\pi_S)$, for any Poisson structure $\pi_S$ on the base. The resulting Lie algebroid is isomorphic to the product $A(P)\times_{TS}T^*S$. The resulting linear model is precisely the Poisson structure obtained in Corollary \ref{corollary:global:form}, where we note that, by \eqref{the_other_bracket}, we have that $\pi_{\ka}=-\piv$. 

These special linear models, coming from principal bundles, are called of \emph{principal type} in \cite{FeMa22}.
\end{example}

We put now these construction into the context of local models around Poisson submanifolds. Consider a Poisson submanifold $(S,\pi_S)$ of a Poisson manifold $(M,\pi_M)$. The restriction to $S$ of the Lie algebroid $T^*M$ fits into a short exact sequence of Lie algebroids:
\begin{equation}\label{eq:short:exact:sequence}
0\rmap \ka\rmap T^*_SM\rmap T^*S\rmap 0,    
\end{equation}
where the bundle of Lie algebras $\ka$ can be identified with the conormal bundle of $S$ in $M$. This short exact sequence is a geometric way of encoding the \emph{first order jet} of $\pi$ along $S$. We call the short exact sequence \textbf{partially split} if there exist:
\begin{itemize}
\item a coupling data $(\nabla,U)$ for $\ka$ and
\item a vector bundle splitting of \eqref{eq:short:exact:sequence} which induces a Lie algebroid isomorphism $T^*_SM\simeq T^*S\oplus \ka$, where $T^*S\oplus \ka$ is endowed with the Lie algebroid structure coming from $(\nabla,U)$.
\end{itemize}
If this is the case, then the corresponding linear model $(O,\pi_0)$, where $S\subset O\subset \ka^*$, is a first order approximation of $\pi$ around $S$, in the sense that 
\begin{itemize}
    \item $(S,\pi_S)$ is a Poisson submanifold of $(O,\pi_0)$;
    \item the short exact sequences of the restricted Lie algebroids are isomorphic.
\end{itemize}
Moreover, if $T^*_SM$ is partially split, then different coupling data produce linear models that are isomorphic around $S$. 

The Poisson manifold $(M,\pi_M)$ is said to be \textbf{linearizable} around the Poisson submanifold $(S,\pi_S)$ if its short exact sequence \eqref{eq:short:exact:sequence} is partially split, and a neighborhood of $S$ in $(M,\pi_M)$ is Poisson-diffeomorphic to a neighborhood of $S$ in the linear model $(O,\pi_0)$. 

Since any linear model is linearizable, using Example \ref{ex:principal:type}, Corollary \ref{corollary:global:form} and the Main Theorem, we obtain Corollary \ref{corollary:linearizable} from the Introduction.

\section{Other values of the moment map}\label{sec : Other values of the moment map}

The Main Theorem has a natural generalization to other values of the moment map. First, we explain the conditions on the value, and some of its consequences.

\subsection*{Conditions on the value of the moment map}
Let $\lambda \in \gg^*$. We denote by $G_{\lambda}$ the stabilizer of $\lambda$ with respect to the coadjoint action, and by $\mathcal{O}_{\lambda}=G\cdot \lambda$ the coadjoint orbit through $\lambda$, which comes equipped with its symplectic form $\omega_{{\lambda}} \in \Omega^2(\mathcal{O}_{\lambda})$. We will assume that $\lambda$ satisfies two assumptions, which we call \emph{split} and \emph{slice}; for a detailed discussion of these conditions, see Section 2.3.1 in \cite{GLS96}.
\begin{itemize}
    \item [($\lambda$1)] \textbf{Split:} the Lie subalgebra $\gg_{\lambda}\subset \gg$ has a $G_{\lambda}$-invariant complement $\mathfrak{c}\subset \gg$. 
\end{itemize}
This condition on $\lambda$ is called \emph{split} in \cite{GLS96}, \emph{reductive} in \cite{Montgomery}, and \emph{Molino's condition} in \cite{Wein85E}. We denote by $q:\gg \to \gg_{\lambda}$ the $G_{\lambda}$-equivariant projection associated with the decomposition $\gg=\gg_{\lambda}\oplus\mathfrak{c}$. We will 
use more often the affine inclusion: 
\begin{equation}\label{eq:j:map}
j:\gg^*_{\lambda}\hookrightarrow \gg^*,\qquad j(\xi)=\lambda+q^*(\xi).
\end{equation}

Since $j(\gg_{\lambda}^*)$ is a transverse and complementary to the orbit $\mathcal{O}_{\lambda}$, it follows that a neighborhood of $\lambda$ in $j(\gg_{\lambda}^*)$ is a Poisson transversal \cite{Frejlich_Marcut}. It is well-known \cite{Molino,Wein85E} that the induced Poisson structure on the Poisson transversal is linearizable around $\lambda$ via the isomorphism $j$. Moreover, $j(\gg_{\lambda}^*)$ is a Poisson-Dirac submanifold \cite[Illustration 2]{PT1.5}, and globally we have that:
\[
j^!\mathrm{Gr}(\pi_{\gg}) = \mathrm{Gr}(\pi_{\gg_{\lambda}}).
\]

The abstract normal bundle of the orbit $\mathcal{O}_{\lambda}$ is $G$-equivariantly isomorphic to $G\times_{G_{\lambda}}\gg^*_{\lambda}$. In fact $j$ induces a $G$-equivariant ``exponential map'':
\begin{align*}
 & \epsilon: G \times_{G_{\lambda}} \gg_{\lambda}^* \rmap \gg^*, & \epsilon([g,\xi]):= g \cdot j(\xi).
\end{align*} 
This map is called  \emph{Molino's exponential map} in \cite[\S 1.3.]{Montgomery}, where the Poisson structure corresponding to $\epsilon^!\mathrm{Gr}(\pi_{\gg})$ is described in detail. In particular, it is shown that this Poisson structure coincides with the construction from Proposition \ref{proposition_coupling_Poisson} for the principal $G_{\lambda}$-bundle $P:=G$ over the symplectic manifold $(S,\pi_S):=(\mathcal{O}_{\lambda},\omega_{\lambda}^{-1})$, and moreover that this Poisson structure blows up exactly at the critical points of $\epsilon$. We prove next a related property: 

\begin{lemma}\label{lemma:split}
The following three properties describe the same $G_{\lambda}$-invariant, open neighborhood 
\[0\in \mathcal{T}_{\mathrm{max}}\subset \gg^{*}_{\lambda}:\]
\begin{enumerate}[i)]
    \item the points $\xi\in \gg_{\lambda}^*$ for which $j(\gg_{\lambda}^*)$ is a Poisson transversal around $j(\xi)$;
    \item the points $\xi\in \gg_{\lambda}^*$ for which $[g,\xi]\in G\times_{G_{\lambda}}\gg^*_{\lambda}$ is a regular point of $\epsilon$, for some (and hence all) $g\in G$;  
    \item the points $\xi\in \gg_{\lambda}^*$ for which $\gg_{j(\xi)}\subset \gg_{\lambda}$. 
\end{enumerate}
\end{lemma}
\begin{proof}
By \cite[Illustration 2]{PT1.5}, we have the following decomposition of $\pi_{\gg}$ along $j(\gg_{\lambda}^*)$:
\begin{equation}\label{P_along_PT}
\pi_{\gg,j(\xi)}(v_1+w_1,v_2+w_2) = \pi_{\gg_{\lambda},\xi}(v_1,v_2)+\pi_{\gg,j(\xi)}(w_1,w_2),
\end{equation}
where $v_1,v_2 \in \gg_{\lambda}$ and $w_1,w_2 \in \mathfrak{c}$. Hence $j(\gg_{\lambda}^*)$ is a Poisson transversal at $j(\xi)$,
\[
T_{j(\xi)}\gg^* = T_{j(\xi)}j(\gg_{\lambda}^*) \oplus \pi^{\sharp}_{\gg,j(\xi)}(N^*j(\gg_{\lambda}^*)),
\]exactly when the following 2-form is non-degenerate:
\begin{align*}
    & \eta_{\xi} \in \wedge^2\mathfrak{c}^*, & \eta_{\xi}(w_1,w_2):=\pi_{\gg,j(\xi)}(w_1,w_2).
\end{align*}

Since $\epsilon$ is $G$-equivariant, $[g,\xi]$ is a regular point of $\epsilon$ if and only if $[1,\xi]$ is a regular point of $\epsilon$. Under the canonical identification 
\[\mathfrak{c}\oplus \mathfrak{c}^{\circ}\simeq T_{[1,\xi]}(G\times_{G_{\lambda}}\gg_{\lambda}^*),\]
the differential of $\epsilon$ becomes:
\begin{align*}
    & \epsilon_*[w,\zeta] = -\act_{\gg^*}(w)+\zeta, & (w,\zeta) \in \mathfrak{c} \times \mathfrak{c}^{\circ}.
\end{align*}
Therefore elements in 
$\ker \epsilon_{*,[1,\xi]}$ are of the form $[w,\act_{\gg^*}(w)]$, with $\act_{\gg^*}(w)\in \mathfrak{c}^{\circ}$. Since $\act_{\gg^*}(w)=\pi_{\gg,j(\xi)}^{\sharp}(w)$, by \eqref{P_along_PT}, this last condition is equivalent to $w\in \ker\eta_{\xi}$. So i) and ii) describe the same set. 
 
 Finally, using that the stabilizer Lie algebra of a point in $\gg^*$ coincides with the Poisson-geometric isotropy Lie algebra, and the explicit description of $\pi_{\gg}$ \eqref{P_along_PT}, we obtain
\[\gg_{j(\xi)}=\ker(\pi_{\gg})_{j(\xi)}=(\gg_{\lambda})_{\xi}\oplus \ker \eta_{\xi}\subset \gg_{\lambda}\oplus\mathfrak{c}.\]
Hence $\gg_{j(\xi)} \subset \gg_{\lambda}$ iff $\eta_{\xi}$ is nondegenerate. Hence items i) and iii) describe the same set. 
\end{proof}

In addition to ($\lambda$1), we will assume also the following, stronger condition on $\lambda$:

\begin{itemize}
    \item[($\lambda$2)] \textbf{Slice:} the point $\lambda$ admits an \emph{(affine) slice}, i.e., there is an open, $G_{\lambda}$-invariant set
    \[0\in \mathcal{T} \subset \gg_{\lambda}^*\]
    with the property that $j(\mathcal{T})$ is a slice for the action --- meaning that, for $g\in G$, \[j(\mathcal{T})\cap (g\cdot j(\mathcal{T}))\neq \emptyset\quad \implies \quad g\in G_{\lambda}.\]
\end{itemize}

It is easy to see that the slice condition is equivalent to the restriction
\begin{equation}\label{eq:exp:rest}
\epsilon: G\times_{G_{\lambda}}\mathcal{T}\rmap \gg^*
\end{equation}
being injective. We also note the following:
\begin{lemma}
Any slice $\mathcal{T}$ is contained in the set $\mathcal{T}_{\mathrm{max}}$ defined in Lemma \ref{lemma:split}. Hence $j(\mathcal{T})$ is a Poisson transversal, and the exponential map $\epsilon$ restricts to a $G$-equivariant diffeomorphism, from $G\times_{G_{\lambda}}\mathcal{T}$ onto its image. In particular, the orbit $\mathcal{O}_{\lambda}$ is an embedded submanifold.
\end{lemma}
\begin{proof}
Take $v\in \gg_{j(\xi)}$, with $\xi\in \mathcal{T}$. Then, for all $t\in \mathbb{R}$, we have that 
\[g_t\cdot j(\xi)=j(\xi)\in j(\mathcal{T})\cap (g_t\cdot j(\mathcal{T})),\] 
where $g_t$ denotes $\exp(t v)$. By the slice condition, it follows that $g_t\in G_{\lambda}$. Hence $\gg_{j(\xi)}\subset \gg_{\lambda}$, so item iii) in Lemma \ref{lemma:split} holds. By item i), $j(\mathcal{T})$ is a Poisson transversal. By item ii) it follows that $\epsilon$ is a local diffeomorphism on $G\times_{G_{\lambda}}\mathcal{T}$, and since it is injective, it is a diffeomorphism onto its image.
\end{proof}

Finally, let us remark that conditions ($\lambda$1) and ($\lambda$2) hold, for example, when $G$ is compact (see Section 2.3.2 in \cite{GLS96} for an explicit slice in this case) and, more generally, when the action of $G$ is proper at $\lambda$. There are other cases when these conditions hold, for example, when $\lambda=0$.

\subsection*{Reducing to the slice}

Consider a Poisson Hamiltonian space
\[
G\acts (M,\pi)\stackrel{\mu}{\rmap}\gg^*
 \]
 and $\lambda\in \gg^*$ a value of the moment map which satisfies conditions ($\lambda$1) and ($\lambda$2). Since $j(\mathcal{T})$ is a Poisson transversal, the Poisson map $\mu:M \to \gg^*$ is transverse to it, its pullback
\[
X:=\mu^{-1}(j(\mathcal{T}))
\]
is again a Poisson transversal in $(M,\pi)$, and the induced map
\begin{align*}
    & \mu_X : (X,\pi_X) \rmap (\gg_{\lambda}^*,\pi_{\gg_{\lambda}}),
    & \mu_X:=j^{-1}\circ \mu|_X
\end{align*}
is again a Poisson map \cite[Lemma 2]{Frejlich_Marcut}.s In fact, 
\begin{equation}\label{eq:restriction:of:Ham:space}
G_{\lambda} \acts (X,\pi_X) \stackrel{\mu_X}{\rmap} \gg_{\lambda}^* 
\end{equation}
is a Poisson Hamiltonian space. In the next subsection, we will make extra assumptions of this Poisson Hamiltonian space, in order to apply the Main Theorem to it. Before that, we show that this Poisson Hamiltonian space determines the original Hamiltonian space around $X$.

\begin{proposition}\label{prop:restr:slice}
Consider a Poisson Hamiltonian space
\begin{equation}
\label{Poisson_Ham_space2}
G\acts (M,\pi)\stackrel{\mu}{\rmap}\gg^*,
 \end{equation}
and a value of the moment map $\lambda\in \gg^*$ which satisfies conditions ($\lambda$1) and ($\lambda$2). 

\begin{enumerate}[a) ]
\item With the notation introduced above, the map
\[\varphi: G\times_{G_{\lambda}}X\rmap M,\quad [g,x]\mapsto g\cdot x\]
is a $G$-equivariant diffeomorphism onto a $G$-invariant neighborhood of $X$ in $M$.
\item Denote by $p_{\lambda}:G\times X\to G\times_{G_{\lambda}}X$ and $p_X:G\times X\to X$ the projections, and by $\pi_0$ be the pullback by $\varphi$ of the Poisson structure $\pi$ on $M$: $\varphi^!\Gr(\pi)=\Gr(\pi_0)$. Then $\pi_0$ is determined by the equality of Dirac structures on $G\times X$: 
\[p^!_{{\lambda}}\Gr(\pi_0)=\mathcal{R}_{\dd \alpha}p_{X}^!\Gr(\pi_X),\]
where $\alpha$ is the one-form given by:
\[\alpha(V,U)=\la j\circ \mu_X(x),\theta_G^l(V) \ra, \qquad (V,U)\in T_{g}G\times T_xX,  \]  
and where $\theta_G^l$ denotes left-invariant Maurer Cartan form on $G$.
\item Hence, equipped the with moment map
\[\mu_0:=\mu\circ \varphi, \quad \mu_0([g,x])=g\cdot (j\circ \mu_X(x)),\]
we obtain a Poisson Hamiltonian space
\[G\acts (G\times_{G_{\lambda}}X,\pi_0)\stackrel{\mu_0}{\rmap}\gg^*,\]
which is isomorphic around $X$ to the Poisson Hamiltonian space \eqref{Poisson_Ham_space2}.
\end{enumerate}
\end{proposition}

\begin{remarks}\normalfont
\begin{enumerate}[(1)]
    \item For a symplectic version of this result, see Corollary 2.3.6 and Theorem 2.3.7 in \cite{GLS96}
\item The reader might have noticed that the map $j\circ \mu_X: X\to \gg^*$ is just the restriction $\mu|_X$. The more cumbersome notation was used in order to emphasize that the construction of the Poisson structure and moment map on $G\times_{G_{\lambda}}X$ depends only on the $G_{\lambda}$-Poisson Hamiltonian space \eqref{eq:restriction:of:Ham:space} and the affine map $j:\gg_{\lambda}^*\hookrightarrow \gg^*$ inducing the slice.
\item The proposition provides a recipe of how to ``extend'' a $G_{\lambda}$-Poisson Hamiltonian space to a $G$-Hamiltonian space, once an affine slice has been fixed. 
\item As opposed to our Main Theorem, this proposition holds also in the holomorphic or algebraic setting, because the isomorphism does not depend on choosing a connection or on the existence of some splitting; it depends only on the existence of the affine slice.
\end{enumerate}
\end{remarks}

\begin{proof}
We start by proving the form from the statement of the pullback of $\mu$:
\[\mu\circ \varphi([g,x])=\mu(g\cdot x)=g\cdot \mu(x)=g\cdot (j\circ \mu_X(x)),\]
where we used $G$-equivariance of $\mu$.
Hence, we have a commutative diagram of $G$-equivariant maps: 
\[
\xymatrix{
G\times_{G_{\lambda}} X \ar[r]^{\varphi}\ar[d]_{\mathrm{id}\times \mu_X}\ar[dr]^{\mu_0} & M \ar[d]^{\mu}\\
G\times_{G_{\lambda}}\mathcal{T}  \ar[r]^{\epsilon} & \gg^*. 
}
\]

Next, we show that $\varphi$ is injective. By commutativity of the diagram and the fact that $\epsilon$ is injective, we have that $\varphi([g,x])=\varphi([h,y])$ implies that $[g,\mu_X(x)]=[h,\mu_X(y)]$ in $G\times_{G_{\lambda}}\mathcal{T}$, and so $h=gk$, with $k\in G_{\lambda}$. Then $\varphi([g,y])=\varphi([h,y])$ implies that $g\cdot x=gk\cdot y$, and so $x=k\cdot y$. Hence $[g,x]=[h,y]$.

We show that $\varphi$ is a diffeomorphism onto its image. Since it is injective, it suffices to show that its differential is invertible, and by $G$-equivariance, it suffices to show this along $X$. Moreover, note that, since $\mu$ is transverse to $\mathcal{T}$, it follows that:
\[\mathrm{codim}_M(X)=\mathrm{codim}_{\gg^*}(j(\mathcal{T}))=\mathrm{codim}_{\gg}(\gg_{\lambda})=\mathrm{codim}_{G}(G_{\lambda}).\]
Hence the dimensions match, and so it suffices to show that $\varphi_*$ is injective at points in $X$. Let $v\in \gg$ and $U\in T_xX$ be such that $[v,U]\in T_{[e,x]}(G\times_{G_{\lambda}}X)$ is in the kernel of $\varphi_*$. Using commutativity of the diagram and that $\epsilon_*$ is invertible, it follows that $[v,(\mu_X)_*U]=0$, hence $v\in \gg_{\lambda}$, so by changing representatives, we may assume that $v=0$. But then $\varphi_*[0,U]=i_*U=0$, where $i:X\hookrightarrow M$ is the inclusion, and so $U=0$. This concludes the proof that $\varphi$ is a diffeomorphism onto its image. 

Next consider the Dirac structure on $G\times X$ given by:
\[D:=p_{\lambda}^! \varphi^! \Gr(\pi).\]
This Dirac structure has the following properties:
\begin{enumerate}[1.]
%    \item The vertical bundle for the projection $p_{\lambda}:G\times X\to G\times_{G_{\lambda}}X$ is included in $D$:
%    \[((l_g)_*v,\act_X(v)_x,0,0)\in D_{(g,x)},\quad \forall\, v\in \gg_{\lambda};\]
    \item $D$ is invariant under the action of $G$
    \begin{align*}
        & G \acts G \times X, & g \cdot (g',x):=(gg',x)
    \end{align*}
    %\[\big((l_g)_*\times  \mathrm{id}_{TX}\times l_{g^{-1} }^*\times \mathrm{id}_{T^*X}\big)\, D_{(h,x)}=D_{(gh,x)}.\]
    \item $D$ is Hamiltonian for this action of $G$, with moment map $\mu_0\circ p_{\lambda}$:
    \[
    G \acts (G \times X,D) \stackrel{\mu_0p_{\lambda}}{\rmap} \gg^*,
    \]that is: for all $v \in \gg$, we have that
    \begin{align*}
        & \act_{G\times X}(v)+(\mu_0p_{\lambda})^*(v) \in \Gamma(D).
    \end{align*}
    %\[(-r_g(v),0,\dd \langle \mu_0 \circ p_{\lambda}, v\rangle) \in D_{(g,x)}, \quad\forall\, v\in \gg;\]
    \item The fibres of $\pr_1 : G \times X \to G$ are Poisson transversals in $(G \times X,D)$, and the induced Poisson structure on the fibres is $\pi_X$.
\end{enumerate}
The first property follows because $\Gr(\pi)$ is $G$-invariant and both $p_{\lambda}$ and $\varphi$ are $G$-equivariant. Because \[ G \acts (G\times_{G_{\lambda}}X,\varphi^!\Gr(\pi)) \stackrel{\mu_0}{\rmap} \gg^*\]
is a Dirac Hamiltonian space, that is,
\[\act_{G\times_{G_{\lambda}}X}(v)+\mu_0^*(v) \in \varphi^!\Gr(\pi)\]
for all $v \in \gg$, and since 
\[(p_{\lambda})_*\act_{G\times X}(v)=\act_{G\times_{G_{\lambda}}X}(v),\]
property 2 follows. Moreover, we obtain an injective vector bundle map
\begin{align*}
        & a : \gg\times G \times X \rmap D, & a(v,g,x):=\act_{G\times X}(v)_{(g,x)}+(\mu_0p_{\lambda})^*(v)_{(g,x)} \in D_{(g,x)}.
    \end{align*}
This implies also that the fibres of $G \times X \to G$ are Dirac transversals for $D$. Because the composition of $\varphi \circ p_{\lambda}:G \times X \to M$ with 
\[i_{1} : X \to G \times X, \qquad i_1(x)=(1,x),\] is just the inclusion $X\hookrightarrow M$, it follows that the fibre $\{1\} \times X$ is a Poisson transversal with induced Poisson structure $\pi_X$. By property 1, the same holds for every fibre $\{g\} \times X$, so property 3 holds. 

Since $(\{1\} \times X,\pi_X)$ is a Poisson transversal in $(G \times X,D)$, for any $\xi\in T^*_xX$ there exists a unique element $b(\xi)\in D_{(1,x)}$ which is $i_1$ related to $\pi_X^{\sharp}\xi+\xi\in \Gr(\pi_X)$. We claim that
\begin{align*}
    b(\xi) = ((j \circ \mu_X)_*\pi_X^{\sharp}(\xi),\pi_X^{\sharp}(\xi)+\xi).
\end{align*}
Note that any element in $\mathbb{T}_{(1,x)}(G\times X)$ which is $i_1$-related to $\pi_X^{\sharp}(\xi)+\xi$ must take the form $b'(\xi):=(\beta_{\xi},\pi_X^{\sharp}(\xi)+\xi)$ for some $\beta_{\xi} \in \gg^*$, and if such an element were to lie in $D_{(1,x)}$, then
\begin{align*}
    \la b'(\xi), a(v)\ra = -\la \beta_{\xi},v\ra +\la (\mu_0p_{\lambda})_*\pi_X^{\sharp}(\xi),v\ra = -\la \beta_{\xi},v\ra +\la (j \circ \mu_X)_*\pi_X^{\sharp}(\xi),v\ra
\end{align*}must vanish for all $v \in \gg$, which is to say that $b'(\xi)=b(\xi)$. %Hence $b(\xi)$ is the unique element of $D_{(1,x)}$ which is $i_1$-related to $\pi_X^{\sharp}(\xi)+\xi \in \mathrm{Gr}(\pi_X)$, for all $\xi \in T^*_xX$ and $x \in X$. 

By property 1, we get an injective vector bundle map
\begin{align*}
    & b : G \times T^*X \rmap D, & b(g,\xi):=g_*b(\xi).
\end{align*}
It is easy to see that $a$ and $b$ are transverse, and by comparing ranks, we obtain an isomorphism:
\begin{align*}
  &\gg \times G \times T^*X \diffto D,   & (v,g,\xi) \mapsto a(v,g,x)+b(g,\xi).
\end{align*}

Next, we prove the formula for $D$ from item b) in the statement:
\begin{align}
D = \mathcal{R}_{\dd\alpha}p_X^!\mathrm{Gr}(\pi_X),
\end{align}
where $\alpha \in \Omega^1(G \times X)$ is the unique one-form with the property that
\begin{align*}
    & \alpha(v^L,U):=\la j \circ \mu_X,v\ra, & (v,U) \in \gg \times T_xX,
\end{align*}where $v^L$ denotes the extension of $v \in \gg$ of a left-invariant vector field on $G$. This one-form is $G$-invariant, and, for all $v_1,v_2\in \gg$ and $U_1,U_2\in \mathfrak{X}(X)$, $\dd\alpha$ reads
\begin{align}\label{eq:d:alpha}
\dd \alpha((v_1^L,U_1),(v_2^L,U_2))&= \la (j\circ \mu_X)_* U_1,v_2\ra -
\la (j\circ \mu_X)_* U_2,v_1\ra-\la j\circ \mu_X ,[v_1,v_2]\ra.
\end{align} 

This implies that
\begin{align*}
\iota_{\pi_X^{\sharp}(\xi)}\dd \alpha_{(1,x)} = (j\circ \mu_X)_*\pi_X^{\sharp}(\xi)_x, 
\end{align*}
and so:
\[b(\xi)=\mathcal{R}_{\dd\alpha}\circ \mathcal{R}_{(0,\pi_X)}(0,\xi)\in 
\mathcal{R}_{\dd\alpha}p_X^!\Gr(\pi_X),
\]
and by property 1, $g_*b(\xi)\in \mathcal{R}_{\dd\alpha}p_X^!\Gr(\pi_X)$. 

In fact, fix $w\in \gg$, $U\in \mathfrak{X}(X)$ and denote by $\phi_{U}^t$ the local flow of $U$. Then, at $(1,x)$, we have that 
\begin{align*}
\la (\mu_0\circ p_{\lambda})^*(v),(w^L,U)\ra_{(1,x)} & = \la (\mu_0\circ p_{\lambda})_*(w,U),v\ra\\
&=\tfrac{\dd}{\dd t}\big|_{t=0}\la \mu_0\circ p_{\lambda}(\exp(tw),\phi_{U}^t(x)),v\ra\\
&=\tfrac{\dd}{\dd t}\big|_{t=0}\la \exp(tw)\cdot j\circ \mu_X(\phi_{U}^t(x)),v\ra\\
&=\tfrac{\dd}{\dd t}\big|_{t=0}\la j\circ \mu_X(\phi_{U}^t(x)), \exp(-tw)\cdot v\ra\\
&=\la (j\circ \mu_X)_*U, v\ra+\la j\circ \mu_X, [v,w]\ra\\
&=\left(-\iota_{v^L}\dd \alpha(w^L,U)\right)_{(1,x)}.
\end{align*}

Because $-(v^L,0)$ coincides with $\act_{G \times X}(v)$ at points of the form $(1,x) \in G \times X$, we have that
\begin{align*}
    a(v)_{(1,x)} = \act_{G \times X}(v)_{(1,x)} + (\mu_0\circ p_{\lambda})_{(1,x)}^*(v) \in \mathcal{R}_{\dd\alpha}p_X^!\mathrm{Gr}(\pi_X)_{(1,x)}.
\end{align*}
Therefore, $G$-invariance of $\mathcal{R}_{\dd\alpha}p_X^!\mathrm{Gr}(\pi_X)$ and $G$-equivariance of the moment map $\mu_0\circ p_{\lambda}$ imply:
\begin{align*}
    a(v)_{(g,x)} = g_*a(g^{-1} \cdot v)_{(1,x)} \in \mathcal{R}_{\dd\alpha}p_X^!\mathrm{Gr}(\pi_X)_{(g,x)}.
\end{align*}
We conclude that the sections $a(v)$ and $b(\xi)$ which span $D$ lie in fact in $\mathcal{R}_{\dd\alpha}p_X^!\mathrm{Gr}(\pi_X)$, from which it follows that these two Dirac structures coincide:
\[
\mathcal{R}_{\dd\alpha}p_X^!\mathrm{Gr}(\pi_X) = D.\qedhere
\]
\end{proof}

\subsection*{The normal form theorem}
Next, we give a normal form theorem around more general values of the moment map. As before, we consider a principal Hamiltonian $G$-space
\begin{equation}\label{Poisson_Ham_space2_again}
G\acts (M,\pi)\stackrel{\mu}{\rmap}\gg^*.
 \end{equation}
For any $\lambda\in \mu(M)$, we have the principal $G_{\lambda}$-bundle $P_{\lambda}:=\mu^{-1}(\lambda)$ over $S_{\lambda}:=\mu^{-1}(\lambda)/G_{\lambda}$.
The base manifold $S_{\lambda}$ has an induced Poisson structure $\pi_{S_{\lambda}}$, characterized by the following condition:
\begin{align*}
& p_{S_{\lambda}}^!\mathrm{Gr}(\pi_{S_{\lambda}})=i^!\mathrm{Gr}(\pi), 
\end{align*}
where $p_{S_\lambda}:P_{\lambda} \to S_{\lambda}$ is the projection and $i:P_{\lambda}\hookrightarrow M$ is the inclusion. The normal form theorem shows that, for special values $\lambda$, the data 
\begin{align*}\label{eq:datum}
& G_{\lambda}\acts P_{\lambda}\stackrel{p_{\lambda}}{\rmap} (S_{\lambda},\pi_{S_{\lambda}})
\end{align*}
can be used to reconstruct the Poisson Hamiltonian $G$-space \eqref{Poisson_Ham_space2_again} around $\mu^{-1}(\lambda)$.

\begin{theorem}\label{thm : normal form other values}
For $\lambda\in \gg^*$ satisfying ($\lambda$1) and ($\lambda$2), any principal Hamiltonian $G$-space \eqref{Poisson_Ham_space2_again} is isomorphic around $P_{\lambda}$ to the Hamiltonian Poisson space
\[
G \acts \left((G\times P_{\lambda}\times \gg_{\lambda}^*)/G_{\lambda},\overline{D}\right)\stackrel{\overline{\mu}}{\rmap} \gg^*,
\]
where the moment map is given by 
\[\overline{\mu}([g,x,\xi])=g\cdot j(\xi),\]
and the Dirac structure $\overline{D}$ is determined by the condition that its pullback via $p_{\lambda}:G\times P_{\lambda}\times \gg_{\lambda}^*\to (G\times P_{\lambda}\times \gg_{\lambda}^*)/G_{\lambda}$ is given by 
\[p_{\lambda}^!\overline{D}=\mathcal{R}_{\dd(\alpha-\widetilde{\theta}_{\lambda})}p_{S_{\lambda}}^!\mathrm{Gr}(\pi_{S_{\lambda}}),\]
where $\widetilde{\theta}_{\lambda} \in \Omega^1(P_{\lambda} \times \gg_{\lambda}^*)$ is the one-form corresponding to a principal  $G_{\lambda}$-connection $\theta_{\lambda} \in \Omega^1(P_{\lambda};\gg_{\lambda})$, and the one-form $\alpha \in \Omega^1(G \times \gg_{\lambda}^*)$ is given by
\begin{align*}
    & \alpha(V,\eta)_{(g,\xi)}=\la j(\xi), (g^{-1})_*V\ra, & (V,\eta) \in T_gG \times T_{\xi}\gg_{\lambda}^*.
\end{align*}
\end{theorem}
%In particular, \eqref{Poisson_Ham_space2_again} is
%determined up to isomorphism around $P_{\lambda}$ by the $G_{\lambda}$-principal bundle $p_{S_\lambda}:P_{\lambda} \to S_{\lambda}$ and the Poisson structure $\pi_{S_{\lambda}}$.

\begin{proof}[Proof of Theorem \ref{thm : normal form other values}]
Using the notation of the previous subsection, we have a Poisson transversal $X:=\mu^{-1}(j(\mathcal{T}))$, and
\begin{equation}\label{eq:restriction:of:Ham:space:2}
G_{\lambda} \acts (X,\pi_X) \stackrel{\mu_X}{\rmap} \gg_{\lambda}^* 
\end{equation}
is a free and proper Poisson Hamiltonian space. Moreover, we have that:
\[\mu_X^{-1}(0)=P_{\lambda},\]
and clearly, the induced Poisson structure on the base $S_{\lambda}$ is $\pi_{S_{\lambda}}$. So we can apply the Main Theorem to \eqref{eq:restriction:of:Ham:space:2}, and we obtain that it is isomorphic around $P_{\lambda}$ to the Dirac Hamiltonian space:
\[
G_{\lambda} \acts \left(P_{\lambda}\times \gg_{\lambda}^*,\mathcal{R}_{-\mathrm{d}\widetilde{\theta}_{\lambda}}p_{S_{\lambda}}^!\mathrm{Gr}(\pi_{S_{\lambda}})\right) \stackrel{\pr_2}{\rmap} \gg_{\lambda}^*,
\]
where $\theta_{\lambda}$ is a $G_{\lambda}$-principal connection on $P_{\lambda}\to S_{\lambda}$. Next, we apply Proposition \ref{prop:restr:slice}, and conclude that a $G$-invariant neighborhood of $P_{\lambda}$ in $M$ is isomorphic, as a Poisson Hamiltonian space, to a $G$-invariant neighborhood of $P_{\lambda}$ in the Dirac Hamiltonian space:
\[
G \acts \left((G\times P_{\lambda}\times \gg_{\lambda}^*)/G_{\lambda},\overline{D}\right)\stackrel{\overline{\mu}}{\rmap} \gg^*,
\]
where the moment map is given by 
\[\overline{\mu}([g,x,\xi])=\epsilon ([g,\xi])=g\cdot j(\xi),\]
and the Dirac structure is determined by the equality of Dirac structure on $G\times P_{\lambda}\times \gg_{\lambda}^*$:
\[p_{\lambda}^!\overline{D}=\mathcal{R}_{\dd(\alpha-\widetilde{\theta}_{\lambda})}p_{S_{\lambda}}^!\mathrm{Gr}(\pi_{S_{\lambda}}),\]
and the 1-forms are given by 
\[\alpha(V,U,\eta)_{(g,x,\xi)}=\la j(\xi), (g^{-1})_*V\ra,\qquad 
\widetilde{\theta}_{\lambda}(V,U,\eta)_{(g,x,\xi)}=\la\xi,{\theta}_{\lambda}(U)\ra.\qedhere\]
\end{proof}

\begin{remark}\normalfont
The extension to proper and infinitesimally free actions holds in this setting as well. 
\end{remark}

\begin{remark}\normalfont
Since the normal form holds on a $G$-invariant neighborhood, it can be seen also as a normal form around the preimage of the orbit \[Q_{\lambda}:=\mu^{-1}(\mathcal{O}_{\lambda})=\mu^{-1}(G\cdot \lambda).\]
The theorem provides a $G$-equivariant identifications: $(G\times P_{\lambda})/G_{\lambda}\simeq Q_{\lambda}$, and of the local model $(G\times P_{\lambda}\times \gg_{\lambda}^*)/G_{\lambda}$ with the normal bundle of $Q_{\lambda}$ in $M$. Moreover, since $\mathcal{O}_{\lambda}\subset \gg^*$ is a Poisson submanifold, it follows that $Q_{\lambda}$ is a coisotropic submanifold in $M$. The theorem implies that under the identification $Q_{\lambda}\simeq (G\times P_{\lambda})/G_{\lambda}$, the induced Dirac structure $\overline{D}_{\lambda}$ on $(G\times P_{\lambda})/G_{\lambda}$ has the property that its pullback to $G\times P_{\lambda}$ is given by
\[\mathcal{R}_{\dd \widetilde{\lambda}}p_{S_{\lambda}}^!\mathrm{Gr}(\pi_{S_{\lambda}}),\]
where $\widetilde{\lambda}\in \Omega^1(G)$ is the left-invariant extension of $\lambda$ to $G$. This follows because, under the zero-section $i:G\times P_{\lambda}\hookrightarrow G\times P_{\lambda}\times \gg^{*}_{\lambda}$, we have that $i^*\alpha= \widetilde{\lambda}$ and $i^*\widetilde{\theta}_{\lambda}
= 0$. Thus $((G\times P_{\lambda})/G_{\lambda}, \overline{D}_{\lambda})$ is a coisotropic Poisson transversal in the local model. Around such submanifolds the the local structure is determined, as shown in \cite[Section 4.3]{Geudens_2020}, and our local model is a more specialized version of this construction, where we also keep track of the $G$-action and the moment map.
\end{remark}

\end{document}